\documentclass[12pt]{article}
\usepackage{amsmath,amsthm,amssymb,amsfonts,bm,mathrsfs}
\usepackage{hyperref}

\normalsize
\newtheorem{theorem}{Theorem}
\newtheorem{definition}{Definition}
\newtheorem{corollary}{Corollary}
\newtheorem{lemma}{Lemma}
\newtheorem{example}{Example}
\newtheorem{conjecture}{Conjecture}
\newtheorem{question}{Question}
\renewcommand{\abstract}{{\noindent\small{\bf Abstract:}\quad}}

\begin{document}
\title{\bf $C^{*}$ algebra and inverse chaos 
\author{\small Luo Lvlin\thanks{E-mail:luoll12@mails.jlu.edu.cn and Adrress:Department of Mathematics,Jilin University, ChangChun,130012,P. R. China.}\quad and \quad\small Hou Bingzhe\thanks{E-mail:houbz@jlu.edu.cn and Adrress:Department of Mathematics,Jilin University, ChangChun,130012,P. R. China.}}\\
}
\date{January 24,~2015}
\maketitle

\begin{abstract}
If an invertible linear dynamical systems is Li-York chaotic or other chaotic, what's about it's inverse dynamics? what's about it's adjoint dynamics?
With this unresolved but basic problems,
this paper will give a criterion for Lebesgue operator on separable Hilbert space.
Also we give a criterion for the adjoint multiplier of Cowen-Douglas functions on $2$-th Hardy space.
Last we give some chaos about scalars perturbation of operator and some examples of invertible bounded linear operator such that $T$ is chaotic but $T^{-1}$ is not.
\\
{\noindent\small{\bf Keywords:}}
inverse, chaos, Hardy space, rooter function, Cowen-Douglas function, Spectrum, $C^{*}$ algebra, Lebesgue operator.
\end{abstract}
\begin{center}{\bf 1.Introduction}\end{center}

The ideas of chaos in connection with a map was introduced by Li T.Y.and his teacher Yorke,J.A.\cite{LiTYYorkeJA1975},
after that there are various definitions of what it means for a map to be chaotic and there is a series of papers on Topological Dynamics and Ergodic Theory about chaos, such as \cite{PeterWalters1982}\cite{IwanikA1989}\cite{JRBrowm1976}\cite{JohnMilnor2006}\cite{MShub1987}.

Following Topological Dynamics,Linear Dynamics is also a rapidly evolving branch of functional analysis,which was probably born in 1982 with the Toronto Ph.D.thesis of C.Kitai \cite{CKitai1982}.
It has become rather popular because of the efforts of many mathematicians,
for the seminal paper \cite{GGodefroy2003} by G.Godefroy and J.H.Shapiro,the notes \cite{JHShapiro2001} by J.H.Shapiro,the authoritative survey \cite{KGGrosseErdmann1999} by K.-G.Grosse-Erdmann,,and finally for the book \cite{FBayartEMatheron2009} by F.Bayart and E.Matheron,the book \cite{KGGrosseErdmannAPerisManguillot} by K.-G.Grosse-Erdmann and A.Peris.

For finite-dimension linear space,the authors have made a topologically conjugate classification about Jordan blocks in \cite{JWRobbin1972}\cite{NHKuiperJWRobbin1973}. So for the eigenvalues $|\lambda|\neq1$, the operators of Jordan block are not Li-Yorke chaotic.With \cite{FBayartEMatheron2009} we know that a Jordan block is not supercyclic when its eigenvalues $|\lambda|=1$, and following a easy discussion it is not Li-Yorke chaotic too.

So the definition of Li-Yorke chaos should be valid only on infinite-dimension Frechet space or Banach space such that in this paper the Hilbert space is infinite-dimensional.
Because a finite-dimensional linear operator could be regard as a compact operator on some Banach spaces or some Hilbert spaces,
we can get the same conclusion from \cite{NilsonCBernardesJrAntonioBonillaVladimirMullerApeiris2012}P12 or
the Theorem $7$ of \cite{HoubingzheTiangengShiluoyi2009}.

For a Frechet space $X$,let $\mathcal{L}(X)$ denote the set of all
bounded linear operators on $X$.
Let $\mathbb{B}$ denote a Banach space and let $\mathbb{H}$ denote a Hilbert space.
If $T\in\mathcal{L}(\mathbb{B})$,
then define $\sigma(T)=\{\lambda\in\mathbb{C};T-\lambda \text{ is not invertible}\}$
and define $r_{\sigma}(T)=\sup\{|\lambda|;\lambda\in\sigma(T)\}$.

\begin{definition}\label{liyorkehundundedingyi1}
Let $T\in\mathcal{L}(\mathbb{B})$,if there exists $x\in\mathbb{B}$ satisfies:

(1)$\varlimsup\limits_{n\to\infty}|T^{n}(x)\|>0$; and
(2)$\varliminf\limits_{n\to\infty}\|T^{n}(x)\|=0$.

Then we say that $T$ is Li-Yorke chaotic,and named $x$ is a Li-Yorke chaotic point of $T$,where $x\in\mathbb{B},n\in\mathbb{N}$.
\end{definition}

Define a distributional function $F_{x}^{n}(\tau)=\frac{1}{n}\sharp\{0\leq i\leq n:\|T^n(x)\|<\tau\}$,
where $T\in\mathcal{L}(\mathbb{B}),x\in\mathbb{B},n\in\mathbb{N}$.
And define

$F_{x}(\tau)=\liminf\limits_{n\to\infty} F_{x}^{n}(\tau)$; and
$F_{x}^{*}(\tau)=\limsup\limits_{n\to\infty} F_{x}^{n}(\tau)$.

\begin{definition}\label{fenbuhundundedingyi2}
Let $T\in\mathcal{L}(\mathbb{B})$,if there exists $x\in\mathbb{B}$ and

$(1)$ If $F_{x}(\tau)=0,\exists\tau>0$,
and $F_{x}^{*}(\epsilon)=1,\forall\epsilon>0$,
then we say that $T$ is distributional chaotic or $I$-distributionally chaotic .

$(2)$ If $F_{x}^{*}(\epsilon)>F_{x}(\tau),\forall\tau>0$,
and $F_{x}^{*}(\epsilon)=1,\forall\epsilon>0$,
then we say that $T$ is $II$-distributionally chaotic.

$(3)$ If $F_{x}^{*}(\epsilon)>F_{x}(\tau),\forall\tau>0$,then we say that $T$ is $III$-distributionally chaotic.
\end{definition}

\begin{definition}[\cite{NilsonCBernardesJrAntonioBonillaVladimirMullerApeiris2012}]\label{liyorkechaoscriteriondingyi0}
Let $X$ is an arbitrary infinite-dimensional separable Frechet space,
$T\in\mathcal{L}(X)$,If there exists a subset $X_0$ of $X$ satisfies:

$(1)$ For any $x\in X_0,\{T^nx\}_{n=1}^{\infty}$ has a subsequence converging to $0$;

$(2)$ There is a bounded sequence $\{a_n\}_{n=1}^{\infty}$ in $\overline{span(X_0)}$ such that the sequence $\{T^na_n\}_{n=1}^{\infty}$ is unbounded.

Then we say $T$ satisfies the Li-Yorke Chaos Criterion.
\end{definition}

\begin{theorem}[\cite{NilsonCBernardesJrAntonioBonillaVladimirMullerApeiris2012}]\label{liyorkechaoscriteriondingli0}
Let $X$ is an arbitrary infinite-dimensional separable Frechet space,
If $T\in\mathcal{L}(X)$,then the following assertions are equivalent.

$(i)$ $T$ is Li-Yorke chaotic;

$(ii)$ $T$ satisfies the Li-Yorke Chaos Criterion.
\end{theorem}

\begin{lemma}[\cite{SRolewicz1969}]\label{youxianweikongjianmeichaoxunhuanxing2}
There are no hypercyclic operators on a finite-dimensional space $X\neq0$.
\end{lemma}

\begin{example}[\cite{FBayartEMatheron2009}P8]\label{chaoxunhuanyudanweikaiyuanpanjiaofeikongdeyinyongyinli34}
Let $\phi\in\mathcal{H}^{\infty}(\mathbb{D})$ and let
$M_{\phi}:\mathcal{H}^{2}(\mathbb{D})\to\mathcal{H}^{2}(\mathbb{D})$ be the associated multiplication operator.
The adjoint multiplier $M_{\phi}^{*}$ is hypercyclic if and only if $\phi$ is non-constant and $\phi(\mathbb{D})\bigcap\mathbb{T}\neq\emptyset$.
\end{example}

\begin{theorem}[\cite{GDBirkhoff1922}]\label{chaoxunhuanxingjiushichuandixing3}
Let $X$ is an arbitrary separable Frechet space,
$T\in\mathcal{L}(X)$.The following assertions are equivalent.

$(i)$ $T$ is hypercyclic.

$(ii)$ $T$ is topologically transitive;
that is,for each pair of non-empty open sets $U,V\subseteq X$,
there exists $n\in\mathbb{N}$ such that $T^n(U)\bigcap V\neq\emptyset$.
\end{theorem}

\begin{theorem}[\cite{FBayartEMatheron2009}]\label{duijiaoxianyujiaquanyiweizhuanzhihunhedingliyinyong11}
Let $X$ is a topological vector space,$T$ is a bounded linear operator on $X$.Let
\begin{eqnarray*}
\left\{\begin{array}{l}
\Lambda_{1}(T)\triangleq
span(\bigcup\limits_{|\lambda|=1,n\in\mathbb{N}}\ker(T-\lambda)^{N}\bigcap ran(T-\lambda)^{N});\\
\Lambda^{+}(T)\triangleq
span(\Lambda_{1}(T)\bigcup\bigcup\limits_{|\lambda|>1,n\in\mathbb{N}}\ker(T-\lambda)^{N});\\
\Lambda^{-}(T)\triangleq
span(\Lambda_{1}(T)\bigcup\bigcup\limits_{|\lambda|<1,n\in\mathbb{N}}\ker(T-\lambda)^{N}).
\end{array}\right.
\end{eqnarray*}

If $\Lambda^{+}(T)$ and $\Lambda^{-}(T)$ are both dense in $X$,then $T$ is mixing.
\end{theorem}

\begin{center}{\bf 2.From Polar Decomposition to functional calculus}\end{center}

The Polar Decomposition Theorem \cite{JohnBConway2000}P15 on Hilbert space is a useful theorem,
especially for invertible bounded linear operator.
We give some properties of $C^{*}$ algebra generated by normal operator.

Let $\mathbb{H}$ be a separable Hilbert space over $\mathbb{C}$
and let $X$ be a compact subset of $\mathbb{C}$.
Let $\mathcal{C}(X)$ denote the linear space of all continuous functions on the compact space $X$,
let $\mathcal{P}(x)$ denote the set of all polynomials on $X$ and
let $T$ be an invertible bounded linear operator on $\mathbb{H}$.
By the Polar Decomposition Theorem \cite{JohnBConway2000}P15 we get $T=U|T|$,
where $U$ is an unitary operator and $|T|^2=T^{*}T$.
Let $\mathcal{A}(|T|)$ denote the $C^{*}$ algebra generated by the positive operator $|T|$ and $1$.

\begin{lemma}\label{weierstrassnikefenyinli1}
Let $0\notin X$ be a compact subset of $\mathbb{C}$.
If $\mathcal{P}(x)$ is dense in $\mathcal{C}(X)$,
then $\mathcal{P}(\frac{1}{x})$ is also dense in $\mathcal{C}(X)$.
\end{lemma}
\begin{proof}
By the property of polynomials we know that $\mathcal{P}(\frac{1}{x})$ is a algebraic closed subalgebra of $\mathcal{C}(X)$ and we get:

$(1)$ $1\in\mathcal{P}(\frac{1}{x})$;

$(2)$ $\mathcal{P}(\frac{1}{x})$ separate the points of $X$;

$(3)$ If $p(\frac{1}{x})\in\mathcal{P}(x)$,then $\bar{p}(\frac{1}{x})\in\mathcal{P}(x)$.

By the Stone-Weierstrass Theorem \cite{JohnBConway1990}P145 we get the conclusion.
\end{proof}

\begin{lemma}\label{weierstrasspingfangkefenyinli2}
Let $X\subseteq\mathbb{R_+}$.
If $\mathcal{P}(|x|)$ is dense in $\mathcal{C}(X)$,
then $\mathcal{P}(|x|^2)$ is also dense in $\mathcal{C}(X)$.
\end{lemma}
\begin{proof}
For $X\subseteq\mathbb{R_+}$,
$x\neq y$
$\Longleftrightarrow x^2\neq y^2$.
By Lemma $\ref{weierstrassnikefenyinli1}$ we get the conclusion.
\end{proof}

By the GNS construction \cite{JohnBConway1990}P250 for the $C^{*}$ algebra $\mathcal{A}(|T|)$,
we get the following decomposition.

\begin{lemma}\label{gnsfenjiedingliyingyongyinli3}
Let $T$ be an invertible bounded linear operator on the separable Hilbert space $\mathbb{H}$ over $\mathbb{C}$,
$\mathcal{A}(|T|)$ is the complex $C^{*}$ algebra generated by $|T|$ and $1$.
There is a sequence of nonzero $\mathcal{A}(|T|)$-invariant subspace.
$\mathbb{H}_1,\mathbb{H}_2,\cdots$ such that:

$(1)$ $\mathbb{H}=\mathbb{H}_1\bigoplus\mathbb{H}_2\bigoplus\cdots$;

$(2)$ For every $\mathbb{H}_i$,
there is a $\mathcal{A}(|T|)$-cyclic vector $\xi^i$ such that $\mathbb{H}_i=\overline{\mathcal{A}(|T|)\xi^i}$ and
$|T|\mathbb{H}_i=\mathbb{H}_i=|T|^{-1}\mathbb{H}_i$.
\end{lemma}
\begin{proof}
By \cite{WilliamArveson2002}P54 we get $(1)$,and $|T|\mathbb{H}_i\subseteq\mathbb{H}_i$,that is $\mathbb{H}_i\subseteq|T|^{-1}\mathbb{H}_i$;
by Lemma $\ref{weierstrassnikefenyinli1}$ we get $|T|^{-1}\mathbb{H}_i\subseteq\mathbb{H}_i$.
Hence we get $|T|\mathbb{H}_i=\mathbb{H}_i=|T|^{-1}\mathbb{H}_i$.
\end{proof}

For $\forall n\in\mathbb{N}$,$T^n$ is invertible when $T$ is invertible.
By the Polar Decomposition Theorem \cite{JohnBConway2000}P15 $T^n=U_n|T^n|$,
where $U_n$ is unitary operator and $|T^n|^2=T^{*n}T^{n}$,
we get the following conclusion.

\begin{lemma}\label{ncignsfenjiedingliyingyongyinli4}
Let $T$ be an invertible bounded linear operator on the separable Hilbert space $\mathbb{H}$ over $\mathbb{C}$,
let $\mathcal{A}(|T^k|)$ be the complex $C^{*}$ algebra generated by $|T^k|$ and $1$ and
let $\mathbb{H}_i^{|T^k|}=\overline{\mathcal{A}(|T^k|)\xi_k^{i}}$ be a sequence of non-zero $\mathcal{A}(|T^k|)$-invariant subspace,
there is a decomposition $\mathbb{H}=\bigoplus_i\mathbb{H}_i^{|T^k|}$,
$\xi_k^{i}\in\mathbb{H},i,k\in\mathbb{N}$.
Given a proper permutation of $\mathbb{H}_i^{|T^k|}$ and $\mathbb{H}_j^{|T^{(k+1)}|}$,
we get $T^{*}\mathbb{H}_i^{|T^k|}=\mathbb{H}_i^{|T^{(k+1)}|}$ and
$T^{-1}\mathbb{H}_i^{|T^k|}=\mathbb{H}_i^{|T^{(k+1)}|}$.
\end{lemma}
\begin{proof}
By Lemma $\ref{weierstrasspingfangkefenyinli2}$,it is enough to prove the conclusion on $\overline{\mathcal{P}(|T^k|^2)\xi_k^{i}}$.
For any given $\xi_k^{i}\in\mathbb{H}=\bigoplus_j\mathbb{H}_j^{|T^{(k+1)}|}$,
there is a unique $j\in\mathbb{N}$ such that $\xi_k^{i}\in\mathbb{H}_j^{|T^{(k+1)}|}$.

$(1)$ Because $T$ is invertible,
for any given $\xi_{k}^{i}$,
there is an unique $\eta_i\in\mathbb{H}_{s}^{|T^{(k+1)}|}$ such that $\eta_i=T^{-1}\xi_{k}^{i}$.
For $\forall p\in\mathcal{P}(|x|^2)$,
we get $p(|T^{(k+1)}|^2)\eta_i= T^{*}p(|T^{k}|^2)\xi_{k}^{i}$.
Hence we get

$\mathbb{H}_{s}^{|T^{(k+1)}|}=\overline{\mathcal{P}(|T^{(k+1)}|^2)\xi_{k+1}^{s}}\supseteq T^{*}\overline{\mathcal{P}(|T^{k}|^2)\xi_{k}^{i}}=T^{*}\mathbb{H}_{i}^{|T^{k}|}$.

$(2)$ Similarly,for any given $\xi_{k+1}^{s}$,
there is an unique $\eta_r\in\mathbb{H}_{r}^{|T^{k}|}$ such that $\eta_r=T\xi_{k+1}^{s}$.
For $\forall p\in\mathcal{P}(|x|^2)$,we get

$p(|T^{k}|^2)\eta_r=p(|T^{k}|^2)T\xi_{k+1}^{s}=T^{*-1}p(|T^{(k+1)}|^2)\xi_{k+1}^{s}$.

Hence we get

$\mathbb{H}_{r}^{|T^{k}|}=\overline{\mathcal{P}(|T^{k}|^2)\xi_{k}^{r}}\supseteq T^{*-1}\overline{\mathcal{P}(|T^{(k+1)}|^2)\xi_{k+1}^{s}}=T^{*-1}\mathbb{H}_{s}^{|T^{(k+1)}|}$.

Let $i=r$,by $(1)(2)$ we get $T^{*-1}\mathbb{H}_{s}^{|T^{(k+1)}|}\subseteq \mathbb{H}_{i}^{|T^{k}|}\subseteq T^{*-1}\mathbb{H}_{j}^{|T^{(k+1)}|}$.

Fixed the order of $\mathbb{H}_{i}^{|T^{k}|}$,
by a proper permutation of $\mathbb{H}_{j}^{|T^{(k+1)}|}$ we get $T^{*}\mathbb{H}_{i}^{|T^{k}|}=\mathbb{H}_{i}^{|T^{(k+1)}|}$.

By Lemma $\ref{weierstrassnikefenyinli1}$ and
$T$ is invertible,we get

$(3)$ For any given $\xi_{k}^{i}$,
there is an unique $\eta_i\in\mathbb{H}_{i}^{|T^{(k+1)}|}$ such that $\eta_i=T^{*}\xi_{k}^{i}$.
For $\forall p\in\mathcal{P}(|x|^{-2})$,
we get $p(|T^{(k+1)}|^{-2})\eta_i= T^{-1}p(|T^{k}|^{-2})\xi_{k}^{i}$.
Hence we get

$\mathbb{H}_{i}^{|T^{(k+1)}|}=\overline{\mathcal{P}(|T^{(k+1)}|^{-2})\xi_{k+1}^{i}}\supseteq T^{-1}\overline{\mathcal{P}(|T^{k}|^{-2})\xi_{k}^{i}}=T^{-1}\mathbb{H}_{i}^{|T^{k}|}$.

$(4)$ For any given $\xi_{k+1}^{i}$,there is an unique $\eta_i\in\mathbb{H}_{i}^{|T^{k}|}$ such that $\eta_i=T^{*-1}\xi_{k+1}^{i}$.
For $\forall p\in\mathcal{P}(|x|^{-2})$,we get

$T^{-1}p(|T^{k}|^{-2})\eta_i= T^{-1}p(|T^{k}|^{-2})T^{*-1}\xi_{k+1}^{i}=p(|T^{(k+1)}|^{-2})\xi_{k+1}^{i}$.

Hence we get

$T^{-1}\mathbb{H}_{i}^{|T^{k}|}=T^{-1}\overline{\mathcal{P}(|T^{k}|^{-2})\xi_{k}^{i}}\supseteq \overline{\mathcal{P}(|T^{(k+1)}|^{-2})\xi_{k+1}^{i}}=\mathbb{H}_{i}^{|T^{(k+1)}|}$.

By $(3)(4)$ we get $T^{-1}\mathbb{H}_{i}^{|T^{k}|}\subseteq \mathbb{H}_{i}^{|T^{(k+1)}|}\subseteq T^{-1}\mathbb{H}_{i}^{|T^{k}|}$.

that is,$T^{-1}\mathbb{H}_{i}^{|T^{k}|}=\mathbb{H}_{i}^{|T^{(k+1)}|}$.
\end{proof}

Let $\xi\in\mathbb{H}$ is a $\mathcal{A}(|T|)$-cyclic vector such that $\mathcal{A}(|T|)\xi$ is dense in $\mathbb{H}$.
Because of $\sigma{|T|}\neq\emptyset$,
on $\mathcal{C}(\sigma(|T|))$ define the non-zero linear functional $\rho_{|T|}$:$\rho_{|T|}(f)=<f(|T|)\xi,\xi>,\forall f\in\mathcal{C}(\sigma(|T|))$.
Then $\rho_{|T|}$ is a positive linear functional,
by \cite{WilliamArveson2002}P54 and the Riesz-Markov Theorem,
on $\mathcal{C}(\sigma(|T|))$ we get that there exists an unique finite positive Borel measure $\mu_{|T|}$ such that
\begin{eqnarray*}
\left.\begin{array}{lr}
\int\limits_{\sigma(|T|)}f(z)\,d\mu_{|T|}(z)=<f(|T|)\xi,\xi>, & \forall f\in\mathcal{C}(\sigma(|T|)).
\end{array}\right.
\end{eqnarray*}

\begin{theorem}\label{hanshuyansuan5}
Let $T$ be an invertible bounded linear operator on the separable Hilbert space $\mathbb{H}$ over $\mathbb{C}$,
there is $\xi\in\mathbb{H}$ such that $\overline{\mathcal{A}(|T|)\xi}=\mathbb{H}$.
For any given $n\in\mathbb{N}$,
let $\mathcal{A}(|T^n|)$ be the complex $C^{*}$ algebra generated by $|T^n|$ and $1$ and
let $\xi_n$ be a $\mathcal{A}(|T^n|)$-cyclic vector such that $\overline{\mathcal{A}(|T^n|)\xi_n}=\mathbb{H}$.
Then:

$(1)$ For any given $\xi_n$,there is an unique positive linear functional
\begin{eqnarray*}
\left.\begin{array}{lr}
\int\limits_{\sigma(|T^n|)}{f(z)\,d\mu_{|T^n|}(z)}=<f(|T^n|)\xi_n,\xi_n>, & \forall f\in\mathcal{L}^{2}(\sigma(|T^n|)).
\end{array}\right.
\end{eqnarray*}

$(2)$ For any given $\xi_n$,there is an unique finite positive complete Borel measure $\mu_{|T^n|}$ such that
$\mathcal{L}^{2}(\sigma(|T^n|))$ is isomorphic to $\mathbb{H}$.
\end{theorem}
\begin{proof}
Because $T$ is invertible,
by Lemma $\ref{ncignsfenjiedingliyingyongyinli4}$ we get that if there is a $\mathcal{A}(|T|)$-cyclic vector $\xi$,
then there is a $\mathcal{A}(|T^n|)$-cyclic vector $\xi_n$.

$(1)$:For any given $\xi_n$, define the linear functional,
$\rho_{|T^n|}(f)=<f(|T^n|)\xi_n,\xi_n>$,
by \cite{WilliamArveson2002}P54 and the Riesz-Markov Theorem we get that on $\mathcal{C}(\sigma(|T^n|))$ there is an unique finite positive Borel measure $\mu_{|T^n|}$ such that
\begin{eqnarray*}
\left.\begin{array}{lr}
\int\limits_{\sigma(|T^n|)}{f(z)\,d\mu_{|T^n|}(z)}=<f(|T^n|)\xi_n,\xi_n>, & \forall f\in\mathcal{C}(\sigma(|T^n|)).
\end{array}\right.
\end{eqnarray*}

Moreover we can complete the Borel measure $\mu_{|T^n|}$ on $\sigma(|T^n|)$,
also using $\mu_{|T^n|}$ to denote the complete Borel measure,
By \cite{PaulRHalmos1974} we know that the complete Borel measure is uniquely.

For $\forall f\in\mathcal{L}^{2}(\sigma(|T^n|))$,because of
\begin{eqnarray*}
\left.\begin{array}{l}
\rho_{|T^n|}(|f|^2)=\rho_{|T^n|}(\bar{f}f)=<f(|T^n|)^{*}f(|T^n|)\xi_n,\xi_n>=\|f(|T^n|)\xi_n\|\geq0.
\end{array}\right.
\end{eqnarray*}

we get that $\rho_{|T^n|}$ is a positive linear functional,hence $(1)$ is right.

$(2)$ we know that $\mathcal{C}(\sigma(|T^n|))$ is a subspace of $\mathcal{L}^{2}(\sigma(|T^n|))$ such that  $\mathcal{C}(\sigma(|T^n|))$ is dense in $\mathcal{L}^{2}(\sigma(|T^n|))$.

For any $f,g\in\mathcal{C}(\sigma(|T^n|))$ we get
\begin{eqnarray*}
\left.\begin{array}{l}
<f(|T^n|)\xi_n,g(|T^n|)\xi_n>_{\mathbb{H}}
=<g(|T^n|)^{*}f(|T^n|)\xi_n,\xi_n>\\
=\rho_{|T^n|}(\bar{g}f)
=\int\limits_{\sigma(|T^n|)}{f(z)\bar{g}(z)\,d\mu_{|T^n|}(z)}
=<f,g>_{\mathcal{L}^{2}(\sigma(|T^n|))}.
\end{array}\right.
\end{eqnarray*}

Therefor $U_0:\mathcal{C}(\sigma(|T^n|))\to\mathbb{H},f(z)\to f(|T^n|)\xi_n$ is a surjection isometry from
$\mathcal{C}(\sigma(|T^n|))$ to $\mathcal{A}(|T^n|)\xi_n$,
also $\mathcal{C}(\sigma(|T^n|))$ and $\mathcal{A}(|T^n|)\xi_n$ is a dense subspace of
$\mathcal{L}^{2}(\sigma(|T^n|))$ and $\mathbb{H}$,respectively.
Because of $U_0$ a closable operator,it closed extension $U:\mathcal{L}^{2}(\sigma(|T^n|))\to\mathbb{H},f(z)\to f(|T^n|)\xi_n$ is a unitary operator.
Hence for any given $\xi_n$,
$U$ is the unique unitary operator induced by the unique finite positive complete Borel measure $\mu_{|T^n|}$ such that
$\mathcal{L}^{2}(\sigma(|T^n|))$ is isomorphic to $\mathbb{H}$.
\end{proof}

By the Polar Decomposition Theorem \cite{JohnBConway2000}P15,
we get $U^{*}T^{*}TU=TT^{*}$ and $U^{*}|T|^{-2}U=|T^{-2}|$ when $T=U|T|$.
In fact,when $T$ is invertible,
we can choose an specially unitary operator such that $|T|^{-1}$ and $|T^{-1}|$ are unitary equivalent.
We give the following unitary equivalent by Theorem $\ref{hanshuyansuan5}$.

\begin{theorem}\label{TjueduizhiniyuTnijueduizhideguanxi15}
Let $T$ be an invertible bounded linear operator on the separable Hilbert space $\mathbb{H}$ over $\mathbb{C}$ and
let $\mathcal{A}(|T|)$ be the complex algebra generated by $|T|$ and $1$.
There is $\sigma(|T|^{-1})=\sigma(|T^{-1}|)$ and we get that $|T|^{-1}$ and $|T^{-1}|$ are unitary equivalent by the unitary operator $F_{xx^{*}}^{\mathbb{H}}$,
more over the unitary operator $F_{xx^{*}}^{\mathbb{H}}$ is induced by an almost everywhere non-zero function $\sqrt{|\phi_{|T|}|}$,where $\sqrt{|\phi_{|T|}|}\in\mathcal{L}^{\infty}(\sigma(|T|),\mu_{|T|})$.
That is,$d\,\mu_{|T^{-1}|}=|\phi_{|T|}|d\,\mu_{|T|^{-1}}$.
\end{theorem}
\begin{proof}
By Lemma $\ref{gnsfenjiedingliyingyongyinli3}$,lose no generally,let $\xi_{|T|}$ is a $\mathcal{A}(|T|)$-cyclic vector such that $\mathbb{H}=\overline{\mathcal{A}(|T|)\xi_{|T|}}$.

$(1)$ Define the function
$F_{x^{-1}}:\mathcal{P}(x)\rightarrow\mathcal{P}(x^{-1}),F_{x^{-1}}(f(x))=f(x^{-1})$,
it is easy to find that $F_{x^{-1}}$ is linear.Because of

$\int\limits_{\sigma(|T|)}{f(z^{-1})d\mu_{|T|}(z)}=<f(|T|^{-1})\xi,\xi>$
$=\int\limits_{\sigma(|T|^{-1})}{f(z)d\mu_{|T|^{-1}}(z)}$.

We get $d\mu_{|T|^{-1}}(z)=|z|^2d\mu_{|T|}(z)$.Hence

$\|F_{x^{-1}}(f(x))\|_{\mathcal{L}^2(\sigma(|T|^{-1}),\mu_{|T|^{-1}})}$

$=\int\limits_{\sigma(|T|^{-1})}{F_{x^{-1}}(f(x))\bar{F}_{x^{-1}}(f(x))d\mu_{|T|^{-1}}(x)}$

$=\int\limits_{\sigma(|T|^{-1})}{f(x^{-1})\bar{f}(x^{-1})d\mu_{|T|^{-1}}(x)}$

$=\int\limits_{\sigma(|T|)}{|z|^2f(z)\bar{f}(z)d\mu_{|T|}(z)}$

$\leq \sup\sigma(|T|)^2\int\limits_{\sigma(|T|)}{f(z)\bar{f}(z)d\mu_{|T|}(x)}$

$\leq\sup\sigma(|T|)^2\|f(z)\|_{\mathcal{L}^2(\sigma(|T|),\mu_{|T|})}$.

So we get $\|F_{x^{-1}}\|\leq\sup\sigma(|T|)^2$,
by the Banach Inverse Mapping Theorem \cite{JohnBConway1990}P91 we get that $F_{x^{-1}}$ is an invertible bounded linear operator from
$\mathcal{L}^2(\sigma(|T|),\mu_{|T|})$ to $\mathcal{L}^2(\sigma(|T|^{-1}),\mu_{|T|^{-1}})$.

Define the operator $F_{x^{-1}}^{\mathbb{H}}:\mathcal{A}(|T|)\xi\rightarrow\mathcal{A}(|T|^{-1})\xi,F(f(|T|)\xi)=f(|T|^{-1})\xi$.

By Lemma $\ref{weierstrassnikefenyinli1}$ and \cite{WilliamArveson2002}P55,
we get that $F_{x^{-1}}^{\mathbb{H}}$ is a bounded linear operator on the Hilbert space $\overline{\mathcal{A}(|T|)\xi}=\mathbb{H}$,
$\|F_{x^{-1}}^{\mathbb{H}}\|\leq\sup\sigma(|T|)^2$.Moreover we get

$\left.
   \begin{array}{rcl}
    \mathbb{H}                    & \underrightarrow{\qquad |T|\qquad }     &\mathbb{H}\\
    F_{x^{-1}}^{\mathbb{H}}\downarrow  &                             & \downarrow F_{x^{-1}}^{\mathbb{H}}\\
    \mathbb{H}                    & \overrightarrow{\qquad |T|^{-1} \qquad} &\mathbb{H}
    \end{array}
 \right.$

$(2)$ Define the function
$F_{xx^{*}}:\mathcal{P}(xx^{*})\rightarrow\mathcal{P}(x^{*}x),F_{xx^{*}}(f(x,x^{*}))=f(x^{*}x)$.
By \cite{Hualookang1949} we get that $F_{xx^{*}}$ is a linear algebraic isomorphic.

For any $x\in \sigma(|T^{-1}|)$,
by Lemma $\ref{weierstrasspingfangkefenyinli2}$ we get that $\mathcal{P}(|x|^2)$ is dense in $\mathcal{C}(|x|)$,
$\mathcal{C}(|x|)$ is dense in $\mathcal{L}^2(\sigma(|T^{-1}|),\mu_{|T^{-1}|})$.
Hence $\mathcal{P}(|x|^2)$ is dense in $\mathcal{L}^2(\sigma(|T^{-1}|),\mu_{|T^{-1}|})$.

With a similarly discussion,
for any $y\in \sigma(|T|^{-1})$,we get that $\mathcal{P}(|y|^2)$ is dense in $\mathcal{L}^2(\sigma(|T|^{-1}),\mu_{|T|^{-1}})$.

So $F_{xx^{*}}$ is an invertible bounded linear operator from $\mathcal{L}^2(\sigma(|T|^{-1}),\mu_{|T|^{-1}})$ to $\mathcal{L}^2(\sigma(|T^{-1}|),\mu_{|T^{-1}|})$.
Therefor $F_{xx^{*}}\circ F_{x^{-1}}$ is an invertible bounded linear operator from $\mathcal{L}^2(\sigma(|T|),\mu_{|T|})$ to $\mathcal{L}^2(\sigma(|T^{-1}|),\mu_{|T^{-1}|})$.

Because of $\lambda\in\sigma(T^{*}T)\Longleftrightarrow\frac{1}{\lambda}\in\sigma(T^{*-1}T^{-1})$,

we get that $\lambda\in\sigma(|T|)\Longleftrightarrow\frac{1}{\lambda}\in\sigma(|T^{-1}|)$,
that is,$\sigma(|T^{-1}|)=\sigma(|T|^{-1})$.

For any $p_n\in\mathcal{P}(\sigma(|T|^{-1}))\subseteq\mathcal{A}(\sigma(|T|^{-1}))$,
because of $T^{*-1}p_n(|T|^{-1})=p_n(|T^{-1}|)T^{*-1}$,
by \cite{JohnBConway2000}P60 we get that $\mathcal{P}(|T|^{-1})$ and $\mathcal{P}(|T^{-1}|)$ are unitary equivalent.
Hence there is an unitary operator $U\in\mathcal{B}(\mathbb{H})$
such that $U\mathcal{P}(|T|^{-1})=\mathcal{P}(|T^{-1}|)U$,
that is,$U\mathcal{A}(|T|^{-1})=\mathcal{A}(|T^{-1}|)U$ and
$U\overline{\mathcal{A}(|T|^{-1})}\xi_{|T|}=\overline{\mathcal{A}(|T^{-1}|)}U\xi_{|T|}$.
If let $\xi_{|T^{-1}|}=U\xi_{|T|}$,then $\xi_{|T^{-1}|}$ is a $\mathcal{A}(|T^{-1}|)$-cyclic vector and  $\overline{\mathcal{A}(|T^{-1}|)}\xi_{|T^{-1}|}=\mathbb{H}$.Because of

$\int\limits_{\sigma(|T|^{-1})}{f(z)d\mu_{|T|^{-1}}(z)}=<f(|T|^{-1})\xi_{|T|},\xi_{|T|}>$
$=\int\limits_{\sigma(|T|)}{f(\frac{1}{z})d\mu_{|T|}(z)}$.

$\int\limits_{\sigma(|T^{-1}|)}{f(z)d\mu_{|T^{-1}|}(z)}=<f(|T^{-1}|)\xi_{|T^{-1}|},\xi_{|T^{-1}|}>$.

We get $[d\mu_{|T^{-1}|}]=[d\mu_{|T|^{-1}}]$,that is,
$d\mu_{|T^{-1}|}$ and $d\mu_{|T|^{-1}}$ are mutually absolutely continuous,
by \cite{JohnBConway1990}IX.3.6Theorem
 and $(1)$ we get that $d\mu_{|T^{-1}|}=|\phi_{|T|}(\frac{1}{z})|d\mu_{|T|^{-1}}=|z|^2|\phi_{|T|}(z)|d\mu_{|T|}$,
where $|\phi_{|T|}(z)|\neq0,a.e.$ and $|\phi_{|T|}(z)|\in\mathcal{L}^{\infty}(\sigma(|T|),\mu_{|T|})$.
So we get

$\|F_{xx^{*}}\circ F_{x^{-1}}(f(x))\|_{\mathcal{L}^2(\sigma(|T^{-1}|),\mu_{|T^{-1}|})}$

$=\int\limits_{\sigma(|T^{-1}|)}{F_{xx^{*}}\circ F_{x^{-1}}(f(x))\overline{F_{xx^{*}}\circ F_{x^{-1}}}(f(x))d\mu_{|T^{-1}|}(x)}$

$=\int\limits_{\sigma(|T^{-1}|)}{f(x^{-1})\bar{f}(x^{-1})d\mu_{|T^{-1}|}(x)}$

$=\int\limits_{\sigma(|T|^{-1})}{|\phi_{|T|}(\frac{1}{x})|f(x^{-1})\bar{f}(x^{-1})d\mu_{|T|^{-1}}(x)}$

$=\int\limits_{\sigma(|T|^{-1})}{F_{x^{-1}}(\sqrt{|\phi_{|T|}(x)|}f(x))F_{x^{-1}}(\sqrt{|\phi_{|T|}(x)|}\bar{f}(x))d\mu_{|T|^{-1}}(x)}$

$=\|\sqrt{|\phi_{|T|}(\frac{1}{x})|}F_{x^{-1}}(f(x))\|_{\mathcal{L}^2(\sigma(|T|^{-1}),\mu_{|T|^{-1}})}$

Hence $F_{xx^{*}}$ is an unitary operator that is induced by $Uf(\frac{1}{x})=\sqrt{|\phi_{|T|}(\frac{1}{x})|}f(\frac{1}{x})$.

Define $F_{xx^{*}}^{\mathbb{H}}$:
$\left\{
   \begin{array}{l}
   \mathcal{A}(|T|^{-2})\xi_{|T|}\rightarrow\mathcal{A}(|T^{-2}|)\xi_{|T^{-1}|},\\
   F_{xx^{*}}^{\mathbb{H}}(f(|T|^{-2})\xi_{|T|})=f(|T^{-2}|)\xi_{|T^{-1}|}.
    \end{array}
 \right.$

Therefor $F_{xx^{*}}^{\mathbb{H}}$ is an unitary operator from $\overline{\mathcal{A}(|T|^{-1})\xi_{|T|}}$ to $\overline{\mathcal{A}(|T^{-1}|)\xi_{|T^{-1}|}}$.
By Lemma $\ref{gnsfenjiedingliyingyongyinli3}$ and \cite{Hualookang1949}
 we get $\overline{\mathcal{A}(|T|^{-1})\xi}$
$=\mathbb{H}=\overline{\mathcal{A}(|T^{-1}|)\xi}$.
That is,$F_{xx^{*}}^{\mathbb{H}}$ is an unitary operator and we get

$\left.
   \begin{array}{rcl}
    \mathbb{H}                    & \underrightarrow{\qquad |T|^{-1}\qquad }     &\mathbb{H}\\
    F_{xx^{*}}^{\mathbb{H}}\downarrow  &                             & \downarrow F_{xx^{*}}^{\mathbb{H}}\\
    \mathbb{H}                    & \overrightarrow{\qquad |T^{-1}| \qquad} &\mathbb{H}
    \end{array}
 \right.$

So $|T|^{-1}$ and $|T^{-1}|$ are unitary equivalent by $F_{xx^{*}}^{\mathbb{H}}$,
the unitary operator $F_{xx^{*}}^{\mathbb{H}}$ is induced by the function $\sqrt{|\phi_{|T|}(\frac{1}{x})|}$.
\end{proof}

\begin{corollary}\label{TjueduizhiyuTxingjueduizhideguanxi16}
Let $T$ be an invertible bounded linear operator on the separable Hilbert space $\mathbb{H}$ over $\mathbb{C}$ and
let $\mathcal{A}(|T|)$ be the complex algebra generated by $|T|$ and $1$.
There is $\sigma(|T|)=\sigma(|T^{*}|)$ and we get that $|T|$ and $|T^{*}|$ are unitary equivalent by the unitary operator $F_{xx^{*}}^{\mathbb{H}}$,
more over the unitary operator $F_{xx^{*}}^{\mathbb{H}}$ is induced by an almost everywhere non-zero function $\sqrt{|\phi_{|T|}|}$,where $\sqrt{|\phi_{|T|}|}\in\mathcal{L}^{\infty}(\sigma(|T|),\mu_{|T|})$.
That is,$d\,\mu_{|T^{*}|}=|\phi_{|T|}|d\,\mu_{|T|}$.
\end{corollary}
\begin{center}{\bf 3.The chaos between $T$ and $T^{*-1}$ for Lebesgue operator}\end{center}

For the example of singular integral in mathematical analysis,
we know that is independent the convergence or the divergence of the weighted integral between  $x$ and $x^{-1}$,
however some times that indeed dependent for a special weighted function.
For $T$ is an invertible bounded operator on the separable Hilbert space $\mathbb{H}$ over $\mathbb{C}$,
we get $0\notin\sigma(|T^n|)\subseteq\mathbb{R}_{+}$.
In the view of the singular integral in mathematical analysis and by Theorem $\ref{hanshuyansuan5}$,
we get that
$T$ and $T^{*-1}$ should not be convergence or divergence at the same time for $T$ is an invertible bounded operator.
Anyway,they should be convergence or divergence at the same time for some special operators.

Therefor we define the Lebesgue operator
and prove that $T$ and $T^{*-1}$ are Li-Yorke chaotic at the same time for $T$ is a Lebesgue operator.
Then we give an example that $T$ is a Lebesgue operator,but not is a normal operator.

Let $dx$ be the Lebesgue measure on $\mathcal{L}^{2}(\mathbb{R}_{+})$,
by Theorem $\ref{hanshuyansuan5}$ we get that $d\mu_{|T^n|}$ is the complete Borel measure and $\mathcal{L}^2(\sigma(|T^n|))$ is a Hilbert space.
If $\exists N>0$,
for $\forall n\geq N,n\in\mathbb{N}$,$d\mu_{|T^n|}$ is absolutely continuity with respect to $dx$,
by the Radon-Nikodym Theorem \cite{JohnBConway1990}P380 there is $f_n\in\mathcal{L}^{1}(\mathbb{R}_{+})$ such that $d\mu_{|T^n|}=f_n(x)\,dx$.

\begin{definition}\label{lebesguesuanzileidingyi6}
Let $T$ be an invertible bounded linear operator on the separable Hilbert space $\mathbb{H}$ over $\mathbb{C}$,
moreover if $T$ satisfies the following assertions:

$(1)$ If $\exists N>0$,for $\forall n\geq N,n\in\mathbb{N}$
\begin{eqnarray*}
\left\{\begin{array}{lr}
d\mu_{|T^n|}=f_n(x)\,dx,& f_n\in\mathcal{L}^{1}(\mathbb{R}_{+}).\\
x^{2}f_n(x)=f_{n}(x^{-1}),& 0<x\leq1.
\end{array}\right.
\end{eqnarray*}

$(2)$ If $\exists N>0$,
for $\forall n\geq N,n\in\mathbb{N}$,there is a $\mathcal{A}(|T^n|)$-cyclic vector $\xi_n$.
And for any given $0\neq x\in\mathbb{H}$ and for any given $0\neq g_n(t)\in\mathcal{L}^2(\sigma(|T^n|))$,
there is an unique $0\neq y\in\mathbb{H}$ such that $y=g_n(|T^n|^{-1})\xi_n$ when $x=g_n(|T^n|)\xi_n$.

Then we say that $T$ is a Lebesgue operator,
let $\mathcal{L}_{Leb}(\mathbb{H})$ denote the set of all Lebesgue operators.
\end{definition}

\begin{theorem}\label{lebesguesuanzihundundeduichengxing7}
Let $T$ be a Lebesgue operator on the separable Hilbert space $\mathbb{H}$ over $\mathbb{C}$,
then $T$ is Li-Yorke chaotic if and only if $T^{*-1}$ is.
\end{theorem}
\begin{proof} We prove the conclusion by two parts.

$(1)$ Let $\mathbb{H}$ be $\mathcal{A}(|T|)$-cyclic,
that is,there is a vector $\xi$ such that $\overline{\mathcal{A}(|T|)\xi}=\mathbb{H}$.
By Lemma $\ref{ncignsfenjiedingliyingyongyinli4}$ we get that if
there is a $\mathcal{A}(|T|)$-cyclic vector $\xi$,
then there is also a $\mathcal{A}(|T^n|)$-cyclic vector $\xi_n$.

Let $x_0$ be a Li-Yorke chaotic point of $T$,
by Theorem $\ref{hanshuyansuan5}$ and the Polar Decomposition Theorem \cite{JohnBConway2000}P15 and by the define of Lebesgue operator,we get that
for enough large $n\in\mathbb{N}$,there are $g_n(x)\in\mathcal{L}^2(\sigma(|T^n|))$, $f_n(x)\in\mathcal{L}^{2}(\mathbb{R}_{+})$ and $y_0\in\mathbb{H}$
such that $x_0=g_n(|T^n|)\xi_n$,$y_0=g_n(|T^n|^{-1})\xi_n$ and $d\mu_{|T^n|}=f_n(x)\,dx$.

$\|T^nx_0\|\\
=<T^{n*}T^nx_0,x_0>\\
=<|T^n|^2g_n(|T^n|)\xi_n,g_n(|T^n|)\xi_n>\\
=<g_n(|T^n|)^{*}|T^n|^2g_n(|T^n|)\xi_n,\xi_n>\\
=\int\limits_{\sigma(|T^n|)}{x^2g_n(x)\bar{g}(x)\,d\mu_{|T^n|}(x)}\\
=\int_{0}^{+\infty}{x^2|g_n(x)|^2f_n(x)\,dx}\\
=\int_{0}^{1}{x^2|g_n(x)|^2f_n(x)\,dx}+\int_{1}^{+\infty}{x^2|g_n(x)|^2f_n(x)\,dx}\\
=\int_{0}^{1}{x^2|g_n(x)|^2f_n(x)\,dx}+\int_{0}^{1}{x^{-4}|g_n(x^{-1})|^2f_n(x^{-1})\,dx}\\
\triangleq
\int_{0}^{1}{|g_n(x)|^2f_n(x^{-1})\,dx}+\int_{0}^{1}{x^{-2}|g_n(x^{-1})|^2f_n(x)\,dx}\\
=\int_{1}^{+\infty}{x^{-2}|g_n(x^{-1})|^2f_n(x)\,dx}+\int_{0}^{1}{x^{-2}|g_n(x^{-1})|^2f_n(x)\,dx}\\
=\int_{0}^{+\infty}{x^{-2}|g_n(x^{-1})|^2f_n(x)\,dx}\\
=\int\limits_{\sigma(|T^n|)}{x^{-2}g_n(x^{-1})\bar{g}_n(x^{-1})\,d\mu_{|T^n|}(x)}\\
=<g_n(|T^n|)^{*}|T^n|^{-2}g_n(|T^n|^{-1})\xi_n,\xi_n>\\
=<|T^n|^{-2}g_n(|T^n|^{-1})\xi_n,g_n(|T^n|^{-1})\xi_n>\\
=<|T^n|^{-2}y_0,y_0>\\
=<T^{-n}T^{-n*}y_0,y_0>\\
=\|T^{*-n}y_0\|.$

Where $\triangleq$ following the define of $f_n(x)$.

$(2)$ If $\mathbb{H}$ is not $\mathcal{A}(|T|)$-cyclic,
by Lemma $\ref{ncignsfenjiedingliyingyongyinli4}$ we get that for $\forall n\in\mathbb{N}$,
there is a decomposition $\mathbb{H}=\bigoplus_i\mathbb{H}_i^{|T^k|}$,
$\xi_k^{i}\in\mathbb{H},i,k\in\mathbb{N}$,
where $\mathbb{H}_i^{|T^k|}=\overline{\mathcal{A}(|T^k|)\xi_k^{i}}$ is a sequence of $\mathcal{A}(|T^k|)$-invariant subspace,
and do $(1)$ for $\mathbb{H}_i^{|T^k|}$.

By $(1)(2)$ we get that $T$ is Li-Yorke chaotic if and only if $T^{*-1}$ is Li-Yorke chaotic.
\end{proof}

\begin{corollary}\label{lebesguesuanzifenbuhundundengjia8}
Let $T$ be a Lebesgue operator on the separable Hilbert space $\mathbb{H}$ over $\mathbb{C}$,
then $T$ is $I$-distributionally chaotic (or $II$-distributionally chaotic or $III$-distributionally chaotic) if and only if $T^{*-1}$ is $I$-distributionally chaotic (or $II$-distributionally chaotic or $III$-distributionally chaotic).
\end{corollary}

\begin{theorem}\label{lebesguesuanzicunzaixing9}
There is an invertible bounded linear operator $T$ on the separable Hilbert space $\mathbb{H}$ over $\mathbb{C}$,
$T$ is Lebesgue operator but not is a normal operator.
\end{theorem}

\begin{proof}
Let $0<a<b<+\infty$,
then $\mathcal{L}^2([a,b])$ is a separable Hilbert space over $\mathbb{R}$,
because any separable Hilbert space over $\mathbb{R}$ can be expanded to a separable Hilbert space over $\mathbb{C}$,
it is enough to prove the conclusion on $\mathcal{L}^2([a,b])$.
We prove the conclusion by six parts:

$(1)$ Let $0<a<1<b<+\infty$,$M=\{[a,\frac{b-a}{2}],[\frac{b-a}{2},b]\}$.
Construct measure preserving transformation on $[a,b]$.

There is a Borel algebra $\xi(M)$ generated by $M$,
define $\Phi:[a,b]\to[a,b]$,
$\Phi([a,\frac{b-a}{2}])=[\frac{b-a}{2},b]$,
$\Phi([\frac{b-a}{2},b])=[a,\frac{b-a}{2}]$.
Then $\Phi$ is an invertible measure preserving transformation on the Borel algebra $\xi(M)$.

By \cite{PeterWalters1982}P63,$U_{\Phi}\neq1$ is a unitary operator induced by $\Phi$,
where $U_{\Phi}$ is the composition $U_{\Phi}h=h\circ\Phi,\forall h\in\mathcal{L}^2([a,b])$ on $\mathcal{L}^2([a,b])$.

$(2)$ Define $M_xh=xh$ on $\mathcal{L}^2([a,b])$,then $M_x$ is an invertible positive operator.

$(3)$ For $f(x)=\frac{|\ln\,x|}{x},x>0$,
define $d\mu=f(x)d\,x$.
then $f(x)$ is continuous and $f(x)>0,a.e.x\in[a,b]$,
hence $d\,\mu$ that is absolutely continuous with respect to $d\,x$ is finite positive complete Borel measure,
therefor $\mathcal{L}^2([a,b],d\,\mu)$ a separable Hilbert space over $\mathbb{R}$.
Moreover $\mathcal{L}^2([a,b])$ and $\mathcal{L}^2([a,b],d\,\mu)$ are unitary equivalent.

$(4)$ Let $T=U_{\Phi}M_x$,we get $T^{*}T=U_{\Phi}TT^{*}U_{\Phi}^{*}$ and $U_{\Phi}\neq1$.
Because of $U_{\Phi}M_{x}\neq M_{x}U_{\Phi}$ and $U_{\Phi}M_{x^2}\neq U_{\Phi}M_{x^2}$,
we get that $T$ is not a normal operator and $\sigma(|T|)=[a,b]$.

$(5)$ The operator $T=U_{\Phi}M_x$ on $\mathcal{L}^2([a,b])$ is corresponding to the operator $T^{\prime}$ on $\mathcal{L}^2([a,b],d\,\mu)$,$T^{\prime}$ is invertible bounded linear operator and is not a normal operator and $\sigma(|T^{\prime}|)=[a,b]$.

$(6)$ From $\int_{a}^{b}x^{n}f(x)d\,x=\int_{a^n}^{b^n}tf(t^{\tfrac{1}{n}})\frac{1}{nt^{\tfrac{n-1}{n}}}d\,t$,
let $f_n(t)=\frac{1}{n}I_{[a^n,b^n]}f(t^{\tfrac{1}{n}})\frac{1}{t^{\tfrac{n-1}{n}}}$,
we get that $f_n(t)$ is continuous and almost everywhere positive,
hence $f_n(t)d\,t$ is a finite positive complete Borel measure.

 For any $E\subseteq\mathbb{R}_{+}$ define $I_{E}=1$ when $x\in E$
else $I_{E}=0$,so $I_{E}$ is the identity function on $E$.

$(i)$ $f_n(t^{-1})=\frac{1}{n}I_{[a^n,b^n]}f(t^{-\tfrac{1}{n}})\frac{1}{t^{-\tfrac{n-1}{n}}}$
$=\frac{1}{n}I_{[a^n,b^n]}\frac{|\ln\,t^{-\tfrac{1}{n}}|}{t^{-\tfrac{1}{n}}}\frac{1}{t^{-\tfrac{n-1}{n}}}$

$=\frac{1}{n}I_{[a^n,b^n]}t|\ln\,t^{\tfrac{1}{n}}|$

$(i)$ $t^2f_n(t)=\frac{1}{n}I_{[a^n,b^n]}f(t^{\tfrac{1}{n}})\frac{t^2}{t^{\tfrac{n-1}{n}}}$
$=\frac{1}{n}I_{[a^n,b^n]}\frac{|\ln\,t^{\tfrac{1}{n}}|}{t^{\tfrac{1}{n}}}\frac{t^2}{t^{\tfrac{n-1}{n}}}$

$=\frac{1}{n}I_{[a^n,b^n]}t{|\ln\,t^{\tfrac{1}{n}}|}$

By $(i)(ii)$ we get $x^2f_{n}(x)=f_{n}(x^{-1})$.

From $\sigma(|T^{\prime n}|)=[a^n,b^n]$ and
$\int_{a^n}^{b^n}t^2f(t^{\tfrac{1}{n}})\frac{1}{nt^{\tfrac{n-1}{n}}}d\,t=\int_{0}^{+\infty}t^2f_{n}(t)dt$,
let $d\,\mu_{|T^{\prime n}|}=f_n(t)d\,t$,then $d\,\mu_{|T^{\prime n}|}$ is the finite positive complete Borel measure.

For any given $0\neq h(x)\in\mathcal{L}^2([a,b])$ we get $0\neq h(x^{-1})\in\mathcal{L}^2([a,b])$.
$I_{[a,b]}$ is a $\mathcal{A}(|M_{x}^{n}|)$-cyclic vector of the multiplication $M_{x}^{n}=M_{x^n}$,
and $I_{[a^n,b^n]}$ is a $\mathcal{A}(|T^{\prime n}|)$-cyclic vector of $|T^{\prime n}|$,
By Definition $\ref{lebesguesuanzileidingyi6}$ we get that $T^{\prime}$ is Lebesgus operator but not is a normal operator.
\end{proof}

\begin{corollary}\label{lebesguesuanzicunzaizhengguisuanzi10}
There is an invertible bounded linear operator $T$ on the separable Hilbert space $\mathbb{H}$ over $\mathbb{C}$,
$T$ is a Lebesgue operator and also is a positive operator.
\end{corollary}

By the cyclic representation of $C^{*}$ algebra and the GNS construction,
also by the functional calculus of invertible bounded linear operator,
we could study the operator by the integral on $\mathbb{R}$.
This way neither change Li-Yorke chaotic nor the computing,
but by the singular integral in mathematical analysis on the theoretical level we should find that
there is a invertible bounded linear operator $T$ that is Li-Yorke chaotic but $T^{-1}$ is not Li-Yorke chaotic.


\begin{center}{\bf 4. Cowen-Douglas function on Hardy space}\end{center}

For $\mathbb{D}=\{z\in\mathbb{C},|z|<1\}$,
if $g$ is an complex analytic function on $\mathbb{D}$ and there is
$\sup\limits_{r<1}\int_{-\pi}^{\pi}|g(re^{i\theta})|^2\,d\theta<+\infty$,
then we denote $g\in\mathcal{H}^2(\mathbb{D})$,
hence $\mathcal{H}^2(\mathbb{D})$ is a Hilbert space with the norm $\|g\|_{\mathbb{H}^2}^2=\sup\limits_{r<1}\int_{-\pi}^{\pi}|g(re^{i\theta})|^2\,\frac{d\theta}{2\pi}$,
especially $\mathcal{H}^2(\mathbb{D})$ is denoted as a Hardy space.
By the completion theory of complex analytic functions,
the Hardy space $\mathcal{H}^2(\mathbb{D})$ is a special  Hilbert space that is relatively easy not only for theoretical but also for computing,
and we should give some properties about adjoint multiplier operators on $\mathcal{H}^2(\mathbb{D})$.

Any given complex analytic function $g$ has a Taylor expansion $g(z)=\sum\limits_{n=0}^{+\infty}a_nz^n$,
so $g\in\mathcal{H}^2(\mathbb{D})$ and $\sum\limits_{n=0}^{+\infty}a_n^2<+\infty$ are naturally isomorphic.
For $\mathbb{T}=\partial\mathbb{D}$, if $\mathcal{L}^2(\mathbb{T})$ denoted the closed span of all Taylor expansions of functions in $\mathcal{L}^2(\mathbb{T})$,
then $\mathcal{H}^{2}(\mathbb{T})$ is a closed subspace of $\mathcal{L}^2(\mathbb{T})$.
From the naturally isomorphic of $\mathcal{H}^{2}(\mathbb{D})$ and $\mathcal{H}^{2}(\mathbb{T})$ by the properties of analytic function,
we denote $\mathcal{H}^{2}(\mathbb{T})$ also as a Hardy space.

By \cite{JohnBConway1990}P6 we get that any Cauchy sequence with the norm $\|\bullet\|_{\mathbb{H}^2}$ on $\mathcal{H}^2(\mathbb{D})$ is a uniformly Cauchy sequence on any closed disk in $\mathbb{D}$,
in particular,we get that the point evaluations $f\to f(z)$ are continuous linear functional on $\mathcal{H}^2(\mathbb{D})$,
by the Riesz Representation Theorem\cite{JohnBConway1990}P13,for any $g(s)\in\mathcal{H}^2(\mathbb{D})$,
there is a unique $f_z(s)\in\mathcal{H}^2(\mathbb{D})$ such that $g(z)=<g(s),f_z(s)>$,
so define that $f_z$ is a reproducing kernel at $z$.

Let $\mathcal{H}^{\infty}(\mathbb{D})$ denote the set of all bounded complex analytic function on $\mathbb{D}$,
for any given $\phi\in\mathcal{H}^{\infty}(\mathbb{D})$,
it is easy to get that $\|\phi\|_{\infty}=\sup\{|\phi(z)|;|z|<1\}$ is a norm on $\phi\in\mathcal{H}^{\infty}(\mathbb{D})$.
for any given $g\in\mathcal{H}^{2}(\mathbb{D})$,
the multiplication operator $M_{\phi}(g)=\phi g$ associated with $\phi$ on $\mathcal{H}^{2}(\mathbb{D})$ is a bounded linear operator,
and by the norm on $\mathcal{H}^{2}(\mathbb{D})$ we get
$\|M_{\phi}(g)\|\leq\|\phi\|_{\infty}\|g\|_{\mathbb{H}^2}$.
If we denote $\mathcal{H}^{\infty}(\mathbb{T})$ as the closed span of all Taylor expansions of functions in $\mathcal{L}^{\infty}(\mathbb{T})$,
then $\mathcal{H}^{2}(\mathbb{T})$ is a closed subspace of $\mathcal{L}^2(\mathbb{T})$ and $\mathcal{H}^{\infty}(\mathbb{D})$ and $\mathcal{H}^{\infty}(\mathbb{T})$ are naturally isomorphic by
the properties of complex analytic functions \cite{HenriCartanyujiarong2008}P55P97 and the Dirichlet Problem
\cite{HenriCartanyujiarong2008}P103.

There are more properties about Hardy space in
\cite{FBayartEMatheron2009}P7,
\cite{JohnBGarnett2007}P48,
\cite{KennethHoffman1962}P39,
\cite{RonaldGDouglas1998}P133 and
\cite{WilliamArveson2002}P106.

\begin{definition}\label{cowendouglassuanzidingyi2}
For a connected open subset $\Omega$ of $\mathbb{C}$,$n\in\mathbb{N}$,
let $\mathcal{B}_n(\Omega)$ denotes the set of all bounded linear operator $T$ on $\mathbb{H}$ that satisfies:

$(a)$ $\Omega\in\sigma(T)=\{\omega\in\mathbb{C}:T-\omega \text{not invertible}\}$;

$(b)$ $ran(T-\omega)=\mathbb{H}$ for $\omega\in\Omega$;

$(c)$ $\bigvee\ker\limits_{\omega\in\Omega}(T-\omega)=\mathbb{H}$;

$(d)$ $\dim\ker(T-\omega)=n$ for $\omega\in\Omega$.

If $T\in\mathcal{B}_n(\Omega)$, then say that $T$ is a Cowen-Douglas operator.
\end{definition}

\begin{theorem}[\cite{BHouPCuiYCao2010}]\label{cowendouglassuanzidingyidevaneyliyrokefenbuqianghunhedingli2}
For a connected open subset $\Omega$ of $\mathbb{C}$,$T\in\mathcal{B}_n(\Omega)$,we get

$(1)$ If $\Omega\bigcap\mathcal{T}\neq\emptyset$,then $T$ is Devaney chaotic.

$(2)$ If $\Omega\bigcap\mathcal{T}\neq\emptyset$,then $T$ is distributionally chaotic.

$(3)$ If $\Omega\bigcap\mathcal{T}\neq\emptyset$,then $T$ strong mixing.
\end{theorem}

\begin{definition}[\cite{RonaldGDouglas1998}P141]\label{waihanshudingyi46}
Let $\mathcal{P}(z)$ be the set of all polynomials about $z$,where $z\in\mathbb{T}$.
Define a function $h(z)\in\mathcal{H}^2(\mathbb{T})$ is an outer function if
$cl[h(z)\mathcal{P}(z)]=\mathcal{H}^2(\mathbb{T})$.
\end{definition}

\begin{lemma}[\cite{RonaldGDouglas1998}P141]\label{waihanshukeyidexingzhi49}
A function $h(z)\in\mathcal{H}^{\infty}(\mathbb{T})$ is invertible on the Banach algebra $\mathcal{H}^{\infty}(\mathbb{T})$,
if and only if $h(z)\in\mathcal{L}^{\infty}(\mathbb{T})$ and $h(z)$ is an outer function.
\end{lemma}

\begin{theorem}[\cite{JohnBGarnett2007}P81]\label{chengfasuanzishimanshedangqiejindangshiwaihanshu45}
Let $\mathcal{P}(z)$ be the set of all polynomials about $z$,where $z\in\mathbb{D}$.
then $h(z)\in\mathcal{H}^2(\mathbb{D})$ is an outer function if and only if $\mathcal{P}(z)h(z)=\{p(z)h(z);p\in\mathcal{P}(z)\}$ is dense in $\mathcal{H}^2(\mathbb{D})$.
\end{theorem}

Let $\phi$ is a non-constant complex analytic function on $\mathbb{D}$,
for any given $z_0\in\mathbb{D}$,
by \cite{HenriCartanyujiarong2008}P29 we get that there exists $\delta_{z_0}>0$,
exists $k_{z_0}\in\mathbb{N}$, when $|z-z_0|<\delta_{z_0}$, there is
\begin{eqnarray*}
\phi(z)-\phi(z_0)=(z-z_0)^{k_{z_0}}h_{z_0}(z)
\end{eqnarray*}

where $h_{z_0}(z)$ is complex analytic on a neighbourhood of $z_0$ and $h_{z_0}(z_0)\neq0$.

\begin{definition}\label{nyejiexihanshudedingyijitoubujiaobu51}
Let $\phi$ is a non-constant complex analytic function on $\mathbb{D}$,
for any given $z_0\in\mathbb{D}$, there exists $\delta_{z_0}>0$ such that
\begin{eqnarray*}
\phi(z)-\phi(z_0)|_{|z-z_0|<\delta_{z_0}}=p_{n_{z_0}}(z)h_{z_0}(z)|_{|z-z_0|<\delta_{z_0}},
\end{eqnarray*}

$h_{z_0}(z)$ is complex analytic on a neighbourhood of $z_0$ and $h_{z_0}(z_0)\neq0$,
$p_{n_{z_0}}(z)$ is a $n_{z_0}$-th polynomial and the $n_{z_0}$-th coefficient is equivalent $1$.
By the Analytic Continuation Theorem \cite{HenriCartanyujiarong2008}P28,
we get that there is a unique complex analytic function $h_{z_0}(z)$ on $\mathbb{D}$ such that
$\phi(z)-\phi(z_0)=p_{n_{z_0}}(z)h_{z_0}(z)$,
then define $h_{z_0}(z)$ is a rooter function of $\phi$ at the point $z_0$.
If for any given $z_0\in\mathbb{D}$,
the rooter function $h_{z_0}(z)$ has non-zero point but the roots of $p_{n_{z_0}}(z)$ are all in $\mathbb{D}$ and $n_{z_0}\in\mathbb{N}$ is a constant on $\mathbb{D}$ that is equivalent $m$,
then define $\phi$ is a $m$-folder complex analytic function on $\mathbb{D}$.
\end{definition}

\begin{definition}\label{youjiejiexihanshuliCowenDouglashanshudingyi47}
Let $\phi(z)\in\mathcal{H}^{\infty}(\mathbb{D}),n\in\mathbb{N}$,
$M_{\phi}$ is the multiplication by $\phi$ on $\mathcal{H}^{2}(\mathbb{D})$.
if the adjoint multiplier $M_{\phi}^{*}\in\mathcal{B}_n(\bar{\phi}(\mathbb{D}))$,
then define $\phi$ is a Cowen-Douglas function.
\end{definition}

By Definition $\ref{youjiejiexihanshuliCowenDouglashanshudingyi47}$ we get that any constant complex analytic function is not a Cowen-Douglas function

\begin{theorem}\label{cowendouglussuanziyujiexichengfasuanzijigezhonghundundeguanxi24}
Let $\phi(z)\in\mathcal{H}^{\infty}(\mathbb{D})$ be a $m$-folder complex analytic function,
$M_{\phi}$ is the multiplication by $\phi$ on $\mathcal{H}^{2}(\mathbb{D})$.
If for any given $z_0\in\mathbb{D}$,
the rooter functions of $\phi$ at $z_0$ is a outer function,
then $\phi$ is a Cowen-Douglas function,
that is,the adjoint multiplier $M_{\phi}^{*}\in\mathcal{B}_m(\bar{\phi}(\mathbb{D}))$.
\end{theorem}
\begin{proof}
By the definition of $m$-folder complex analytic function Definition $\ref{nyejiexihanshudedingyijitoubujiaobu51}$ we get that $\phi$ is not a constant complex analytic function.
For any given $z\in\mathbb{D}$,
let $f_z\in\mathcal{H}^2(\mathbb{D})$ be the reproducing kernel at $z$.
We confirm that $M_{\phi}^{*}$ is valid the conditions of Definition $\ref{cowendouglassuanzidingyi2}$ one by one.

$(1)$ For any given $z\in\mathbb{D}$,
$f_z$ is an eigenvector of $M_{\phi}^{*}$ with associated eigenvalue $\lambda=\bar{\phi}(z)$.Because for any $g\in\mathcal{H}^2(\mathbb{D})$ we get
\begin{eqnarray*}
<g,M_{\phi}^{*}(f_z)>_{\mathcal{H}^2}=<\phi g,(f_z)>_{\mathcal{H}^2}=\phi(z)f(z)=<g,\bar{\phi}(z)f_z>_{\mathcal{H}^2}
\end{eqnarray*}

By the Riesz Representation Theorem\cite{JohnBConway1990}P13 of bounded linear functional in the form of inner product on Hilbert space, we get $M_{\phi}^{*}(f_z)=\bar{\phi}(z)f_z=\lambda f_z$,
that is,$f_z$ is an eigenvector of $M_{\phi}^{*}$ with associated eigenvalue$\lambda=\bar{\phi}(z)$.

$(2)$ For any given $\bar{\lambda}\in\mathbb{\phi(D)}$,
because of $0\neq\phi\in\mathcal{H}^{\infty}(\mathbb{D})$,
we get that the multiplication operator $M_{\phi}-\lambda$ is injection by the properties of complex analysis,
hence $\ker({M_{\phi}-\lambda})=0$.
Because of $\mathcal{H}^2(\mathbb{D})=\ker(M_{\phi}-\lambda)^{\bot}=cl[ran(M_{\phi}^{*}-\bar{\lambda})]$,
we get that $ran(M_{\phi}^{*}-\bar{\lambda})$ is a second category space.
By \cite{JohnBConway1990}P305 we get that $M_{\phi}^{*}-\bar{\lambda}$ is a closed operator,
also by \cite{JohnBConway1990}P93 or \cite{zhanggongqinglinyuanqu2006}P97 we get
$ran(M_{\phi}^{*}-\bar{\lambda})=\mathcal{H}^2(\mathbb{D})$.

$(3)$ Suppose that $span\{f_z;z\in\phi(\mathbb{D})\}=span\{\frac{1}{1-\bar{z}s};z\in\phi(\mathbb{D})\}$
is not dense in $\mathcal{H}^2(\mathbb{D})$,
By the definition of reproducing kernel $f_z$ and because of $0\neq\phi\in\mathcal{H}^{\infty}(\mathbb{D})$,
we get that there exists $0\neq g\in\mathcal{H}^2(\mathbb{D})$,for any given $z\in\mathbb{D}$, we have
\begin{eqnarray*}
0=<g,\bar{\phi}(z)f_z>_{\mathcal{H}^2}=\phi(z)g(z)=<\phi(z)g(z),f_z>_{\mathcal{H}^2}.
\end{eqnarray*}

So we get $g=0$ by the Analytic Continuation Theorem \cite{HenriCartanyujiarong2008}P28,that is a contradiction for $g\neq0$.
Therefor we get that $span\{f_z;z\in\phi(\mathbb{D})\}$
is dense in $\mathcal{H}^2(\mathbb{D})$,
that is,$\bigvee\ker\limits_{\bar{\lambda}\in\phi(\mathbb{D})}(M_{\phi}^{*}-\bar{\lambda})=\mathcal{H}^2(\mathbb{D})$.

$(4)$ By Definition $\ref{nyejiexihanshudedingyijitoubujiaobu51}$ and the conditions of this theorem,
for any given $\lambda\in\phi(\mathbb{D})$,
there exists $z_0\in\mathbb{D}$, exists $m$-th polynomial $p_m(z)$ and outer function $h(z)$ such that
\begin{eqnarray*}
\left.
 \begin{array}{l}
\phi(z)-\lambda=\phi(z)-\phi(z_0)=p_m(z)h(z),
 \end{array}
 \right.
\end{eqnarray*}

We give $\dim\ker(M_{\phi(z)}^{*}-\bar{\lambda})=m$ by the following $(i)(ii)(iii)$ assertions.

$(i)$ Let the roots of $p_m(z)$ are $z_0,z_1,\cdots,z_{m-1}$,
then there exists decomposition $p_m(z)=(z-z_0)(z-z_1)\cdots(z-z_{m-1})$,
and denote $p_{m,z_0z_1\cdots z_{m-1}}(z)$ is the decomposition of $p_m(z)$ by the permutation of $z_0,z_1,\cdots,z_{m-1}$,
the following to get $\dim\ker M_{p_{m,z_0z_1\cdots z_{m-1}}}^{*}=m$.

By the Taylor expansions of functions in $\mathcal{H}^2(\mathbb{T})$,
we get there is a naturally isomorphic
\begin{eqnarray*}
\left.
 \begin{array}{l}
F_{s}:\mathcal{H}^2(\mathbb{D})\to\mathcal{H}^2(\mathbb{D}-s),F_{s}(g(z))\to g(z+s),\text{ÆäÖÐ}s\in\mathbb{C}.
 \end{array}
 \right.
\end{eqnarray*}

It is easy to get that $G=\{F_s;s\in\mathbb{C}\}$ is a abelian group by the composite operation $\circ$,
hence for $0\leq n\leq m-1$,there is
\begin{displaymath}
\begin{array}{rcl}
\mathcal{H}^2(\mathbb(D)) & \underrightarrow{\qquad\quad~M_{z-z_n}~~\quad\qquad} & \mathcal{H}^2(\mathbb(D))\\
F_{z_n}\downarrow &      & \downarrow F_{z_n}\\
\mathcal{H}^2(\mathbb{D}-z_n)&\overrightarrow{\qquad\qquad~M_{z}^{'}\qquad\qquad}& \mathcal{H}^2(\mathbb{D}-z_n)
\end{array}
\end{displaymath}

Let $T$ is the backward shift operator on the Hilbert space $\mathcal{L}^2(\mathbb{N})$,
that is, $T(x_1,x_2,\cdots)=(x_2,x_3,\cdots)$.
With the naturally isomorphic between $\mathcal{H}^2(\mathbb{D}-z_n)$ and $\mathcal{H}^2(\partial(\mathbb{D}-z_n))$,
$M_{z}^{'*}$ is equivalent the backward shift operator $T$ on $\mathcal{H}^2(\partial(\mathbb{D}-z_n))$,
that is,$M_{z}^{'*}$ is a surjection and $\dim\ker M_{z}^{'*}=1$,
hence $M_{z-z_n}^{*}$ is a surjection and $\dim\ker M_{z-z_n}^{*}=1$, where $0\leq n\leq m-1$.

By the composition of $F_{z_{m-1}}\circ F_{z_{m-2}}\circ\cdots\circ F_{z_{0}}$,
$M_{p_{m,z_0z_1\cdots z_{m-1}}}^{'*}$ is equivalent $T^m$.
that is,$M_{p_{m,z_0z_1\cdots z_{m-1}}}^{'*}$ is a surjection and $\dim\ker M_{p_{m,z_0z_1\cdots z_{m-1}}}^{'*}=m$,
hence $M_{p_{m,z_0z_1\cdots z_{m-1}}}^{*}$ is a surjection and $\dim\ker M_{p_{m,z_0z_1\cdots z_{m-1}}}^{*}=m$.

$(ii)$ Because $\mathcal{H}^{\infty}$ is a abelian Banach algebra,
$M_{p_{m}}$ is independent to the permutation of $1$-th factors of $p_m(z)$,
that is,$M_{p_{m}}^{*}$ is independent to the $1$-th factors multiplication of $p_m(z)=(z-z_0)(z-z_1)\cdots(z-z_{m-1})$.

Because $G=\{F_s;s\in\mathbb{C}\}$ is a abelian group by composition operation $\circ$,
for $0\leq n\leq m-1$,
$F_{z_{m-1}}\circ F_{z_{m-2}}\circ\cdots\circ F_{z_{0}}$ is independent to the permutation of composition.
Hence $M_{p_{m}}^{*}$ is a surjection and
\begin{eqnarray*}
\left.
 \begin{array}{l}
\dim\ker M_{p_{m}}^{*}=\dim\ker M_{p_{m,z_0z_1\cdots z_{m-1}}}^{*}=m.
 \end{array}
 \right.
\end{eqnarray*}

$(iii)$ By Definition $\ref{waihanshudingyi46}$ and
Theorem $\ref{chengfasuanzishimanshedangqiejindangshiwaihanshu45}$,
also by \cite{JohnBConway1990}P93 or \cite{zhanggongqinglinyuanqu2006}P97 and
by \cite{JohnBConway1990}P305 we get that the multiplication operator $M_{h}$ is surjection that associated with the outer function $h$.
Hence we get
\begin{eqnarray*}
\left.
 \begin{array}{l}
\ker M_{h(z)}^{*}=(ranM_{h(z)})^{\bot}=(\mathcal{H}^2(\mathbb{D}))^{\bot}=0.
 \end{array}
 \right.
\end{eqnarray*}

Because there exists decomposition $M_{p_m(z)h(z)}^{*}=M_{h(z)}^{*}M_{p_m(z)}^{*}$ on $\mathcal{H}^2(\mathbb{D})$, we get
\begin{eqnarray*}
\left.
 \begin{array}{l}
\dim\ker(M_{\phi}^{*}-\bar{\lambda})=\dim\ker M_{p_m(z)h(z)}^{*}=\dim\ker M_{p_m(z)}^{*}=m.
 \end{array}
 \right.
\end{eqnarray*}

By $(1)(2)(3)(4)$ we get the adjoint multiplier operator $M_{\phi}^{*}\in\mathcal{B}_m(\bar{\phi}(\mathbb{D}))$.
\end{proof}

By Theorem $\ref{cowendouglussuanziyujiexichengfasuanzijigezhonghundundeguanxi24}$
and Lemma $\ref{waihanshukeyidexingzhi49}$,
we get
\begin{corollary}\label{nyejiexihanshushicowendouglashanshudetiaojianqigenhanshukeni52}
Let $\phi\in\mathcal{H}^{\infty}(\mathbb{D})$ is a $m$-folder complex analytic function,
for any given $z_0\in\mathbb{D}$,
if the rooter function of $\phi$ at $z_0$ is invertible in the Banach algebra $\mathcal{H}^{\infty}(\mathbb{D})$,
then $\phi$ is a Cowen-Douglas function.
Especially,for any given $n\in\mathbb{D}$,if $a$ and $b$ are both non-zero complex,
then $a+bz^n\in\mathcal{H}^{\infty}(\mathbb{D})$ is a Cowen-Douglas function.
\end{corollary}

The following gives some properties about the adjoint multiplier of Cowen-Douglas functions.

\begin{theorem}\label{hardykongjianshangchengfasuanziyuyuanzhoujiaofeikongdedingli17}
If $\phi\in\mathcal{H}^{\infty}(\mathbb{D})$ is a Cowen-Douglas function,
$M_{\phi}$ is the multiplication by $\phi$ on $\mathcal{H}^{2}(\mathbb{D})$,
Then the following assertions are equivalent

$(1)$ $M_{\phi}^{*}$ is Devaney chaotic;

$(2)$ $M_{\phi}^{*}$ is distributionally chaotic;

$(3)$ $M_{\phi}^{*}$ is strong mixing;

$(4)$ $M_{\phi}^{*}$ is Li-Yorke chaotic;

$(5)$ $M_{\phi}^{*}$ is hypercyclic;

$(6)$ $\phi(\mathbb{D})\bigcap\mathbb{T}\neq\emptyset$.
\end{theorem}
\begin{proof}
By Example $\ref{chaoxunhuanyudanweikaiyuanpanjiaofeikongdeyinyongyinli34}$ we get that $M_{\phi}^{*}$ is hypercyclic if and only if $\phi$ is non-constant and $\phi(\mathbb{D})\bigcap\mathbb{T}\neq\emptyset$,
hence $(6)$ is equivalent to $(5)$.

First to get that $(6)$ imply $(1)(2)(3)(4)$.

Because $\phi\in\mathcal{H}^{2}(\mathbb{D})$ is a Cowen-Douglas function,
by Definition $\ref{youjiejiexihanshuliCowenDouglashanshudingyi47}$,
$M_{\phi}^{*}\in\mathcal{B}_n(\bar{\phi}(\mathbb{D}))$.

By Theorem $\ref{cowendouglassuanzidingyidevaneyliyrokefenbuqianghunhedingli2}$
we get that if $\phi(\mathbb{D})\bigcap\mathbb{T}\neq\emptyset$,
then $(1)(2)(3)$ is valid.
On Banach spaces Devaney chaotic, distributionally chaotic and strong mixing imply Li-Yorke chaotic,
respectively.Hence $(4)$ is valid.Because $\bar{\phi}(\mathbb{D})\bigcap\mathbb{T}\neq\emptyset$ and
$\phi(\mathbb{D})\bigcap\mathbb{T}\neq\emptyset$ are mutually equivalent, $(6)$ imply $(1)(2)(3)(4)$.

Then to get that $(1)(2)(3)(4)$ imply $(6)$.By $(1)(2)(3)$ imply $(4)$,respectively,
it is enough to get that $(4)$ imply $(6)$.

If $M_{\phi}^{*}$ is Li-Yorke chaotic,
then we get that $\phi$ is non-constant and by \cite{HouBLiaoGCaoY2012}Theorem3.5 we get $\sup\limits_{n\to+\infty}\|M_{\phi}^{*n}\|\to\infty$,
hence $\|M_{\phi}\|=\|M_{\phi}^{*}\|>1$,
that is,$\sup\limits_{z\in\mathbb{D}}|\phi(z)|>1$.
Moreover,we also have $\inf\limits_{z\in\mathbb{D}}|\phi(z)|<1$,
Indeed,if we assume that
$\inf\limits_{z\in\mathbb{D}}|\phi(z)|\geq1$
then
$\frac{1}{\phi}\in\mathcal{H}^{\infty}$ and $\|M_{\frac{1}{\phi}}^{*}\|=\|M_{\frac{1}{\phi}}\|\leq1$.
Hence for any $0\neq x\in\mathcal{H}^2(\mathbb{D})$ we get
$\|M_{\phi}^{*n}x\|\geq\frac{1}{\|M_{\phi}^{*-n}\|}\|x\|\geq\frac{1}{\|M_{\frac{1}{\phi}}^{*}\|^n}\|x\|\geq\|x\|$.
It is a contradiction to $M_{\phi}^{*}$ is Li-Yorke chaotic.

Therefor that $M_{\phi}^{*}$ is Li-Yorke chaotic imply
$\inf\limits_{z\in\mathbb{D}}|\phi(z)|<1<\sup\limits_{z\in\mathbb{D}}|\phi(z)|$,
By the properties of a simple connectedness argument of complex analytic functions we get $\phi(\mathbb{D})\bigcap\mathbb{T}\neq\emptyset$.
Hence we get $(1)(2)(3)(4)$ both imply $(6)$.
\end{proof}

\begin{corollary}\label{cowendouglussuanzideniyuhardykongjiandechengfasuanzi23}
If $\phi\in\mathcal{H}^{\infty}(\mathbb{D})$ is a invertible Cowen-Douglas function in the Banach algebra $\mathcal{H}^{\infty}(\mathbb{D})$,
and let $M_{\phi}$ be the multiplication by $\phi$ on $\mathcal{H}^{2}(\mathbb{D})$.
Then $M_{\phi}^{*}$ is
Devaney chaotic or distributionally chaotic or strong mixing or Li-Yorke chaotic if and only if $M_{\phi}^{*-1}$ is.
\end{corollary}
\begin{proof}
Because of $T=(T^{-1})^{-1}$,
it is enough to prove that $M_{\phi}^{*}$ is
Devaney chaotic or distributionally chaotic or strong mixing or Li-Yorke chaotic imply
$M_{\phi}^{*-1}$ is.

By Definition $\ref{youjiejiexihanshuliCowenDouglashanshudingyi47}$ we get
$M_{\phi}^{*}\in\mathcal{B}_n(\bar{\phi}(\mathbb{D}))$,
with a simple computing we get $M_{\phi}^{*-1}\in\mathcal{B}_n(\frac{1}{\bar{\phi}}(\mathbb{D}))$.

If $M_{\phi}^{*}$ is
Devaney chaotic or distributionally chaotic or strong mixing or Li-Yorke chaotic,
by Theoem $\ref{hardykongjianshangchengfasuanziyuyuanzhoujiaofeikongdedingli17}$ we get $\phi(\mathbb{D})\bigcap\mathbb{T}\neq\emptyset$,
and by the properties of complex analytic functions we get $\frac{1}{\phi}(\mathbb{D})\bigcap\mathbb{T}\neq\emptyset$.

Because of $M_{\phi}^{*-1}\in\mathcal{B}_1(\frac{1}{\bar{\phi}}(\mathbb{D}))$ and by Theorem $\ref{hardykongjianshangchengfasuanziyuyuanzhoujiaofeikongdedingli17}$ we get
$M_{\phi}^{*-1}$ is
Devaney chaotic or distributionally chaotic or strong mixing or Li-Yorke chaotic.
\end{proof}


\begin{center}{\bf 5. The chaos of scalars perturbation of an operator}\end{center}

We now study some properties about scalars perturbation of an operator inspired by \cite{HoubingzheTiangengShiluoyi2009} and \cite{BermudezBonillaMartinezGimenezPeiris2011} that research some properties about the compact perturbation of scalar operator.

\begin{definition}\label{shuzhijiafasuanzidedingyi1}
Let $\lambda\in\mathbb{C}$,$T\in\mathcal{L}(\mathbb{H})$.Define

$(i)$ Let $S_{LY}(T)$ denote the set such that $\lambda I+T$ is Li-Yorke chaotic for every $\lambda\in S_{LY}(T)$.

$(ii)$ Let $S_{DC}(T)$ denote the set such that $\lambda I+T$ is distributionally chaotic for every $\lambda\in S_{DC}(T)$.

$(iii)$ Let $S_{DV}(T)$ denote the set such that $\lambda I+T$ is Devaney chaotic for every $\lambda\in S_{DV}(T)$.

$(iv)$ Let $S_{H}(T)$ denote the set such that $\lambda I+T$ is hypercyclic for every $\lambda\in S_{H}(T)$.
\end{definition}

By Definition $\ref{shuzhijiafasuanzidedingyi1}$ we get $S_{LY}(\lambda I+T)=\lambda+S_{LY}(T)$,
$S_{DC}(\lambda I+T)=\lambda+S_{DC}(T)$,
$S_{DV}(\lambda I+T)=\lambda+S_{DV}(T)$ and $S_{H}(\lambda I+T)=\lambda+S_{H}(T)$.

\begin{lemma}\label{zhengguisuanzishuzhipuweikong3}
Let $T\in\mathcal{L}(\mathbb{H})$ be a normal operator,then $S_{LY}(T)=\emptyset$.
\end{lemma}
\begin{proof}
Because $T$ is a normal operator,$\lambda I+T$ is a normal operator,too.
by \cite{HoubingzheTiangengShiluoyi2009} we get $S_{LY}(T)=\emptyset$.
\end{proof}

\begin{lemma}\label{jinmilingsuanzishuzhipu4}
There is a quasinilpotent compact operator $T\in\mathcal{L}(\mathbb{H})$
such that $S_{LY}(T)=\mathbb{T}$ is closed and $S_{LY}(T^{*})=\emptyset$,where $\mathbb{T}=\partial\mathbb{D},\mathbb{D}=\{z\in\mathbb{C},|z|<1\}$.
\end{lemma}
\begin{proof}
Let $\{e_n\}_{n\in\mathbb{N}}$ be a orthonormal basis of $\mathcal{L}^2(\mathbb{N})$
and let $T$ be a weighted backward shift operator with weight sequence $\{\omega_n=\frac{1}{n}\}_{n=1}^{+\infty}$
such that $S_{\omega}(e_0)=0$,
$S_{\omega}(e_n)=\omega_{n}e_{n-1}$, where $0<|\omega_n|<M<+\infty$,$\forall n>0$.

By the Spectral Radius formula\cite{JohnBConway1990}P197 $r_{\sigma}(T)=\lim\limits_{n\to+\infty}\|T^n\|^{\frac{1}{n}}$ we get $\sigma(T)=\{0\}$,
hence $\sigma(\lambda I+T)=\lambda$.

$(1)$ If $|\lambda|<1$,
we can select $\varepsilon>0$ such that $|\lambda|+\varepsilon<q<1$.
Then $\forall 0\neq x\in\mathbb{H}$ we get
\begin{eqnarray*}
\lim\limits_{n\to\infty}\|(\lambda I+T)^nx\|
\leq\lim\limits_{n\to\infty}\|(\lambda I+T)^n\|\|x\|
\leq\lim\limits_{n\to\infty}(|\lambda|+\varepsilon)^n\|x\|=0.
\end{eqnarray*}

$(2)$ If $|\lambda|>1$,because of $\sigma((\lambda I+T)^{-1})=\frac{1}{\lambda}$,then we get
\begin{eqnarray*}
\lim\limits_{n\to\infty}\|(\lambda I+T)^nx\|
\geq\lim\limits_{n\to\infty}\frac{1}{\|(\lambda I+T)^{-n}\|}\|x\|
\geq\lim\limits_{n\to\infty}\frac{1}{\|(\lambda I+T)^{-1}\|^n}\|x\|
\geq\|x\|\neq0.
\end{eqnarray*}

$(3)$ By Theorem $\ref{duijiaoxianyujiaquanyiweizhuanzhihunhedingliyinyong11}$ we get that if $|\lambda|=1$,
then $\lambda I+T$ is mixing.
Mixing imply Li-Yorke chaotic.

$(4)$ Because of $\sigma(T)=\sigma(T^{*})$,
by $(1)(2)$ we get that if $|\lambda|\neq1$,then $\lambda+T^{*}$ is not Li-Yorke chaotic.

$(5)$ If $|\lambda|=1$,from the view of infinite matrix $\lambda I+T^{*}$ is lower triangular matrix,
then with a simple computing,for any $0\neq x\in\mathcal{L}^2(\mathbb{N})$,
we get $\varliminf\limits_{n\to\infty}\|\lambda I+S_{\omega}^{*}x\|>0$.
Hence $\lambda I+T^{*}$ is not Li-Yorke chaotic.

By $(1)(2)(3)(4)(5)$ we get that $S_{LY}(T)=\mathbb{T}$ is closed and $S_{LY}(T^{*})=\emptyset$.
\end{proof}

\begin{lemma}\label{shuzhipukeyishikaiji5}
Let $T$ be the backward shift operator on $\mathcal{L}^2(\mathbb{N})$,
$T(x_1,x_2,\cdots)$
$=(x_2,x_3,\cdots)$.
Then $S_{LY}(T)=S_{DC}(T)=S_{DV}(T)=S_{H}(T)=2\mathbb{D}\setminus\{0\}$,
$S_{LY}(2T)=S_{DC}(2T)=S_{DV}(2T)=S_{H}(2T)=3\mathbb{D}$,
Hence $S_{LY}(T)$ and $S_{LY}(2T)$ are open sets.
\end{lemma}
\begin{proof}
By \cite{JohnBConway1990}P209 we get $\sigma(T)=cl\mathbb{D}$ and $\sigma(2T)=cl2\mathbb{D}$,
by Definition $\sigma(T)$ we get $\sigma(\lambda I+T)=\lambda+cl\mathbb{D}$.
Because of the method to prove the conclusion is similarly for $T$ and $2T$,
we only to prove the conclusion for $T$.

By the naturally isomorphic between $\mathcal{H}^{2}(\mathbb{T})$ and $\mathcal{H}^{2}(\mathbb{D})$.
Let $\mathcal{L}^2(\mathbb{N})=\mathcal{H}^{2}(\mathbb{T})$,
by the definition of $T$ we get $(\lambda I+T)^{*}$ is the multiplication operator $M_{f}$ by $f(z)=\bar{\lambda}+z$
on the Hardy space $\mathcal{H}^{2}(\mathbb{T})$.
By the Dirichlet Problem \cite{HenriCartanyujiarong2008}P103 we get that $f(z)$ is associated with the complex analytic function $\phi(z)=\bar{\lambda}+z\in\mathcal{H}^{\infty}(\mathbb{D})$ determined by the boundary condition $\phi(z)|_{\mathbb{T}}=f(z)$.

By Corollary $\ref{nyejiexihanshushicowendouglashanshudetiaojianqigenhanshukeni52}$ we get that $\phi$ is a Cowen-Douglas function.
Therefor by the natural isomorphic between $\mathcal{H}^{2}(\mathbb{T})$ and $\mathcal{H}^{2}(\mathbb{D})$,
$\lambda I+T$ is naturally equivalent to the operator $M_{\phi}^{*}$ on $\mathcal{H}^{2}(\mathbb{D})$.

By Theorem $\ref{hardykongjianshangchengfasuanziyuyuanzhoujiaofeikongdedingli17}$
we get that $M_{\phi}^{*}$ is hypercyclic or Devaney chaotic or distributionally chaotic or Li-Yorke chaotic if and only if $\phi(\mathbb{D})\bigcap\mathbb{T}\neq\emptyset$.

Because of $\sigma(\lambda I+T)=\sigma(\bar{\lambda}I+T^{*})$,
we get $\sigma(\lambda I+T)=\sigma(M_{\phi}^{*})=\sigma(M_{\phi})\supseteq\phi(\mathbb{D})$,
hence $S_{LY}(T)=S_{DC}(T)=S_{DV}(T)=S_{H}(T)=2\mathbb{D}\setminus\{0\}$ is an open set.
\end{proof}

Therefor we can get
\begin{corollary}
Let $T$ be the backward shift operator on $\mathcal{L}^2(\mathbb{N})$,
$T(x_1,x_2,\cdots)$
$=(x_2,x_3,\cdots)$.
For $\lambda\neq0,a\neq0,n\in\mathbb{N}$, if $\lambda+aT^n$ is a invertible bounded linear operator,
then $\lambda+aT^n$ is strong mixing or Devaney chaotic or distributionally chaotic or Li-Yorke chaotic if and only if $(\lambda+aT^n)^{-1}$ is.
\end{corollary}

\begin{theorem}\label{jibushikaijiyebushibiji6}
There is $T\in\mathcal{L}(\mathbb{H})$,
$S_{LY}(T)$ is neither open nor closed.
\end{theorem}
\begin{proof}
Let $T_1,T_2\in\mathcal{L}(\mathbb{H})$,
because $T_1$ or $T_2$ is Li-Yorke chaotic if and only if $T_1\bigoplus T_2$ is,
we get $S_{LY}(T_1\bigoplus T_2)=S_{LY}(T_1)\bigcup S_{LY}(T_2)$.

By Lemma $\ref{jinmilingsuanzishuzhipu4}$,Lemma $\ref{shuzhipukeyishikaiji5}$
and Definition $\ref{shuzhijiafasuanzidedingyi1}$ we get the conclusion.
\end{proof}


\begin{center}{\bf 6. Examples that $T$ is chaotic but $T^{-1}$ is not}\end{center}

In the last we give some examples to confirm the theory giving by functional calculus on the begin that $T$ is chaotic but $T^{-1}$ is not.

\begin{example}\label{duijiaoxianyujiaquanyiweizhuanzhibuhundun12}
Let $\{e_n\}_{n\in\mathbb{N}}$ be a orthonormal basis of $\mathcal{L}^2(\mathbb{N})$
and let $S_{\omega}$ be a backward shift operator on $\mathcal{L}^2(\mathbb{N})$
with weight sequence $\omega=\{\omega_n\}_{n\geq1}$ such that $S_{\omega}(e_0)=0$,
$S_{\omega}(e_n)=\omega_{n}e_{n-1}$,where $0<|\omega_n|<M<+\infty$,
$\forall n>0$.

$(1)$ If $|\lambda|=1$,then $\lambda I+S_{\omega}$ is Li-Yorke chaotic,
but $\lambda I+S_{\omega}^{*}$ and $(\lambda I+S_{\omega}^{*})^{-1}$ are not Li-Yorke chaotic.

$(2)$ Let ${(\lambda I+S_{\omega})}^{n}=U_{n}|{(\lambda I+S_{\omega})}^{n}|$ is the polar decomposition of ${(\lambda I+S_{\omega})}^{n}$,
$\{U_{n}\}_{n=1}^{\infty}$ is not a constant sequence.
\end{example}
\begin{proof}
$(1)$ By Theorem $\ref{duijiaoxianyujiaquanyiweizhuanzhihunhedingliyinyong11}$,
we get that for $|\lambda|=1$,$\lambda I+S_{\omega}$ is mixing and mixing imply Li-Yorke chaotic,
hence $\lambda I+S_{\omega}$ is Li-Yorke chaotic.
From the view of infinite matrix,$\lambda I+S_{\omega}^{*}$ and $(\lambda I+S_{\omega}^{*})^{-1}$ are lower triangular matrix,
with a simple computing,for any $0\neq x\in\mathcal{L}^2(\mathbb{N})$, we get
$\varliminf\limits_{n\to\infty}\|\lambda I+S_{\omega}^{*}x\|>0$,and
$\varliminf\limits_{n\to\infty}\|(\lambda I+S_{\omega}^{*})^{-1}x\|>0$.
Hence $\lambda I+S_{\omega}^{*}$ and $(\lambda I+S_{\omega}^{*})^{-1}$ are not Li-Yorke chaotic.

$(2)$ Let ${(\lambda I+S_{\omega})}^{n}=U_{n}|{(\lambda I+S_{\omega})}^{n}|$ is the polar decomposition of ${(\lambda I+S_{\omega})}^{n}$.
If $\{U_{n}\}_{n=1}^{\infty}$ is a constant sequence,
then by Theorem $\ref{hanshuyansuan5}$ we get that $(\lambda I+S_{\omega}^{*})^{-1}$ is Li-Yorke chaotic.
A contradiction.
Hence $\{U_{n}\}_{n=1}^{\infty}$ is not a constant sequence.
\end{proof}

\begin{theorem}[\cite{HoubingzheTiangengShiluoyi2009}]\label{duijiaoxianyujxiaojinsuanzidingliyinyong13}
For any $\varepsilon>0$,
there is a small compact operator $K_\varepsilon\in\mathcal{L}(\mathbb{H})$ and
$\|K_\varepsilon\|<\varepsilon$ such that $I+K_\varepsilon$ is distributionally chaotic.
\end{theorem}

In \cite{HoubingzheTiangengShiluoyi2009},
$I+K_\varepsilon=\bigoplus\limits_{j=1}^{+\infty}(I_j+K_j)$ is distributionally chaotic.
where $\mathbb{H}=\bigoplus\limits_{j=1}^{+\infty}\mathbb{H}_j$,
$n_j=2m_j$,$\mathbb{H}_j$ is the $n_j$-dimension subspace of $\mathbb{H}$.
On $\mathbb{H}_j$ define:

$S_j=\left[
   \begin{array}{cccc}
    0 & 2\varepsilon_j &       &              \\
      & \ddots         &\ddots &              \\
      &                &\ddots &2\varepsilon_j\\
      &                &       &0
    \end{array}
 \right]_{n_j\times n_j}$,
$K_j=\left[
   \begin{array}{cccc}
    -\varepsilon_j& 2\varepsilon_j &       &  \\
      & \ddots         &\ddots &              \\
      &                &\ddots &2\varepsilon_j\\
      &                &       &-\varepsilon_j
    \end{array}
 \right]_{n_j\times n_j}$.

$I_j+K_j=\left[
   \begin{array}{cccc}
    1-\varepsilon_j& 2\varepsilon_j &       &  \\
      & \ddots         &\ddots &              \\
      &                &\ddots &2\varepsilon_j\\
      &                &       &1-\varepsilon_j
    \end{array}
 \right]_{n_j\times n_j}=(1-\varepsilon_j)I_j+S_j$.

We can construct a invertible bounded linear operator $I+K_\varepsilon$ in the same way that is Li-Yorke chaotic,
but $(I+K_\varepsilon)^{-1}$ is not.

\begin{example}\label{duijiaoxianyujxiaojinsuanzidenibuhundun14}
There is a invertible bounded linear operator $I+K_\varepsilon$ on $\mathbb{H}=\mathcal{L}^2(\mathbb{N})$ such that $I+K_\varepsilon$ is Li-Yorke chaotic,
but $(I+K_\varepsilon)^{-1}$,$(I+K_\varepsilon)^{*-1}$ and $(I+K_\varepsilon)^{*}$ are not Li-Yorke chaotic.
\end{example}
\begin{proof}
Let $\{e_i\}_{i=1}^{\infty}$ is a orthonormal basis of $\mathbb{H}=\mathcal{L}^2(\mathbb{N})$ and
Let $\mathbb{H}=\bigoplus\limits_{j=1}^{+\infty}\mathbb{H}_j$,
$j\in\mathbb{N}$,
where $\mathbb{H}_{j}=\overline{span\{e_i\}},1+\tfrac{j(j-1)}{2}\leq i\leq\tfrac{j(j+1)}{2}$,
$\mathbb{H}_j$ is $j$-dimension subspace of $\mathbb{H}$.
For any given positive sequence $\{\varepsilon_j\}_{1}^{\infty}$ such that
$\varepsilon_j\to0$ and $\sup\limits_{j\to\infty}(1+\varepsilon_j)^{j}\to+\infty$,
on $\mathbb{H}_j$ define:

\begin{eqnarray*}
S_j=\left[
   \begin{array}{cccc}
    0 & 2\varepsilon_j &       &              \\
      & \ddots         &\ddots &              \\
      &                &\ddots &2\varepsilon_j\\
      &                &       &0
    \end{array}
 \right]_{j\times j},
K_j=\left[
   \begin{array}{cccc}
    -\varepsilon_j& 2\varepsilon_j &       &  \\
      & \ddots         &\ddots &              \\
      &                &\ddots &2\varepsilon_j\\
      &                &       &-\varepsilon_j
    \end{array}
 \right]_{j\times j}
 \end{eqnarray*}.

\begin{eqnarray*}
I_j+K_j=\left[
   \begin{array}{cccc}
    1-\varepsilon_j& 2\varepsilon_j &       &  \\
      & \ddots         &\ddots &              \\
      &                &\ddots &2\varepsilon_j\\
      &                &       &1-\varepsilon_j
    \end{array}
 \right]_{j\times j}=(1-\varepsilon_j)I_j+S_j
  \end{eqnarray*}.

First to prove that $(I+K_\varepsilon)$ is Li-Yorke chaotic.

Let $I+K_\varepsilon=\bigoplus\limits_{j=1}^{+\infty}(I+K_j)$,
$f_j=\frac{1}{\sqrt{j}}(1_1,\cdots,1_j)$ and $f_{j,n}=\frac{1}{\sqrt{j}}(1_1,\cdots,1_n,0,\cdots,0)\in\mathbb{H}_j$.
for any $1\leq n\leq j$ we get

$\left.
   \begin{array}{l}
\|(I+K_\varepsilon)^n(f_j)\|\\
=\|(I_j+K_j)^n(f_j)\|\\
=\|((1-\varepsilon_j)I_j+S_j)^n(f_j)\|\\
=\|\sum\limits_{k=0}^{n}C_{n}^{k}(1-\varepsilon_j)^{k}S_j^{n-k}f_j\|\\
\geq\|\sum\limits_{k=0}^{n}C_{n}^{k}(1-\varepsilon_j)^{k}(2\varepsilon_j)^{n-k}(1_1,\cdots,1_n,0,\cdots,0)\|\\
=(1+\varepsilon_j)^n\|f_{j,n}\|.
\end{array}\right.$

Hence we get

$(a)$ $\varlimsup\limits_{j\to\infty}\|(I+K_\varepsilon)^j(f_j)\|\geq\lim\limits_{n\to\infty}\|f_j\|(1+\varepsilon_j)^j=+\infty$.

Because of $r_{\sigma}(I+K_\varepsilon)<1$,we get

$(b)$ $\lim\limits_{n\to\infty}\|(I+K_\varepsilon)^n(f_j)\|=0$.

By $(a)(b)$ and by Definition $\ref{liyorkechaoscriteriondingyi0}$ we get that $\lambda I+K_\varepsilon$ satisfies the Li-Yorke Chaos Criterion,
by Theorem $\ref{liyorkechaoscriteriondingli0}$ we get that $\lambda I+K_\varepsilon$ is Li-Yorke chaotic.

Then to prove that $(I+K_\varepsilon)^{-1}$ is not Li-Yorke chaotic.
For convenience we define $n?^m$ by induction on $m$ for any given $n\in\mathbb{N}$.

For any given $j\in\mathbb{N}$,define:

$(1)$ $j?=1+2+\cdots+j$,

$(2)$ If defined $j?^{n}$ ,
then define $j?^{n+1}=1?^n+2?^n+\cdots+j?^n$.

Let $A=\bigoplus\limits_{j=1}^{+\infty}A_j$,where
$A_j=(I_j+K_j)^{-1}$.Because of $A_j(I_j+K_j)=A_j(I_j+K_j)=I_j$,
we get $A(I+K_\varepsilon)=(I+K_\varepsilon)A=\bigoplus\limits_{j=1}^{+\infty}I_j=I$.
By the Banach Inverse Mapping Theorem \cite{JohnBConway1990}P91 we get that $A=(I+K_\varepsilon)^{-1}$ is a bounded linear operator,
that is,$A=(I+K_\varepsilon)^{-1}$.Hence we get

\begin{eqnarray*}
A_j=\frac{1}{1-\varepsilon_j}\left[
   \begin{array}{cccccc}
   1  & -2\varepsilon_j &       &(-2)^{j-1}\varepsilon_{j}^{j-1}  \\
      & \ddots         &\ddots &              \\
      &                &\ddots &-2\varepsilon_j\\
      &                &       &1
    \end{array}
 \right]_{j\times j}
\end{eqnarray*}.

\begin{eqnarray*}
A_j^{2}
=\frac{1}{(1-\varepsilon_j)^2}\left[
   \begin{array}{cccccc}
   1  & 2\cdot(-2)\varepsilon_j &       &j\cdot(-2)^{j-1}\varepsilon_{j}^{j-1}  \\
      & \ddots         &\ddots &              \\
      &                &\ddots &2\cdot(-2)\varepsilon_j\\
      &                &       &1
    \end{array}
 \right]_{j\times j}
\end{eqnarray*}.

\begin{eqnarray*}
A_j^{3}
=\frac{1}{(1-\varepsilon_j)^3}\left[
   \begin{array}{cccccc}
    1  & (1+2)(-2)\varepsilon_j &       &j?\cdot(-2)^{j-1}\varepsilon_{j}^{j-1}  \\
      & \ddots         &\ddots &              \\
      &                &\ddots &2?(-2)\varepsilon_j\\
      &                &       &1
    \end{array}
 \right]_{j\times j}
\end{eqnarray*}.

For $m\geq3,m\in\mathbb{N}$,If defined

\begin{eqnarray*}
A_j^{m}
=\frac{1}{(1-\varepsilon_j)^m}\left[
   \begin{array}{cccc}
    1 & 2?^{m-2}(-2)\varepsilon_j &       &j?^{m-2}(-2)^{j-1}\varepsilon_j^{j-1}  \\
      & \ddots                     &\ddots &              \\
      &                            &\ddots &2?^{m-2}(-2)\varepsilon_j\\
      &                            &       &1
    \end{array}
 \right]_{j\times j}
\end{eqnarray*}.

Then define

\begin{eqnarray*}
A_j^{(m+1)}=AA^m
=\frac{1}{(1-\varepsilon_j)^{(m+1)}}
\left[
   \begin{array}{cccccc}
     1 & 2?^{m-1}(-2)\varepsilon_j &       &j?^{m-1}(-2)^{j-1}\varepsilon_j^{j-1} \\
       & \ddots                         &\ddots &              \\
       &                                &\ddots &2?^{m-1}(-2)\varepsilon_j\\
       &                                &       &1
    \end{array}
 \right]_{j\times j}
\end{eqnarray*}.

For any given $0\neq x_0=(x_1,x_2,\cdots,)\in\mathbb{H}$,
Let
\begin{eqnarray*}
\left.\begin{array}{l}
y_j=(x_{(1+\frac{j(j-1)}{2})},\cdots,x_{\frac{j(j+1)}{2}});\\
y_j^{\prime}=(x_{(2+\tfrac{j(j-1)}{2})},\cdots,x_{(\frac{j(j+1)}{2}-1)});\\
z_{j,m}=x_{(1+\frac{j(j-1)}{2})}+\sum\limits_{k=2}^{j}k?^{(m-2)}(-2)^{k-1}\varepsilon_j^{k}.
\end{array}\right.
\end{eqnarray*}.

Following a brilliant idea of Zermelo,we shall give the conclusion by induction.

$(1)$ If $y_1\neq0$,then we get

$\varliminf\limits_{n\to\infty}\|A^nx_0\|=\varliminf\limits_{n\to\infty}\sum\limits_{j=1}^{+\infty}\|A_j^ny_j\|\geq$
$\varliminf\limits_{n\to\infty}\|A_1^ny_1\|=\varliminf\limits_{n\to\infty}\frac{|x_1|}{(1-\varepsilon_1)^n}=+\infty$.

$(2)$ If $y_1=0$,but $y_2=(x_2,x_3)\neq0$.

$(i)$ If $x_3\neq0$,
by $(1)$ we get

$\varliminf\limits_{n\to\infty}\|A^nx_0\|\geq\varliminf\limits_{n\to\infty}\frac{|x_3|}{(1-\varepsilon_1)^n}\to+\infty$.

$(ii)$ If $x_3=0$ and $\varepsilon_2>\frac{1}{2?^{(n-2)}}$,because of $y_2=(x_2,x_3)\neq0$,we get

$\varliminf\limits_{n\to\infty}\|A^nx_0\|=\varliminf\limits_{n\to\infty}\sum\limits_{j=1}^{+\infty}\|A_j^ny_j\|$
$\geq\varliminf\limits_{n\to\infty}\|A_2^ny_2\|$

$=\varliminf\limits_{n\to\infty}\frac{1}{(1-\varepsilon_1)^n}\sqrt{(x_2^2+(2?^{n-2}(-2)\varepsilon_2x_3)^2)}$
$\geq\varliminf\limits_{n\to\infty}\frac{|x_2|}{(1-\varepsilon_1)^n}$
$=+\infty$.

$(3)$ Assume for $k\leq m-1$,
there is $\varliminf\limits_{n\to\infty}\|A^nx_0\|\to+\infty$ for $A^{k}$ and $y_k\neq0$.
Then for $k=m$ and $y_m\neq0$ we get

$(i)$ If $x_{m?}\neq0$,
by $(1)$ we get

$\varliminf\limits_{n\to\infty}\|A^nx_0\|\geq\varliminf\limits_{n\to\infty}\frac{|x_{m?}|}{(1-\varepsilon_1)^n}=+\infty$.

$(ii)$ If $x_{m?}=0$ and $\varepsilon_m>\frac{1}{m?^{(n-2)}}$,because of $y_m\neq0$,we get

$\varliminf\limits_{n\to\infty}\|A^nx_0\|=\varliminf\limits_{n\to\infty}\sum\limits_{j=1}^{+\infty}\|A_j^ny_j\|$
$\geq\varliminf\limits_{n\to\infty}\|A_m^ny_m\|$

$=\varliminf\limits_{n\to\infty}\frac{1}{(1-\varepsilon_1)^n}$
$\sqrt{z_{m,n}^2+\|A_{m-1}^ny_m^{\prime}\|^2}$.

If $y_m^{\prime}\neq0$,by the induction hypothesis we get

$\varliminf\limits_{n\to\infty}\|A^nx_0\|\geq\varliminf\limits_{n\to\infty}\|A_{m-1}^ny_m^{\prime}\|=+\infty$;

If $y_m^{\prime}=0$,because of $y_m\neq0$ we get $x_{(1+\frac{m(m-1)}{2})}\neq0$.
by $(1)$ we get

$\varliminf\limits_{n\to\infty}\|A^nx_0\|\geq$
$\varliminf\limits_{n\to\infty}\frac{\|z_{m,n}\|}{(1-\varepsilon_1)^n}$
$=\varliminf\limits_{n\to\infty}\frac{|x_{(1+\frac{m(m-1)}{2})}|}{(1-\varepsilon_1)^n}=+\infty$.

Therefor for $k=m$ and $y_m\neq0$,
we get $\varliminf\limits_{n\to\infty}\|A^nx_0\|\to+\infty$,
by the induction we get that for any $m\in\mathbb{N}$ and $y_m\neq0$,there is  $\varliminf\limits_{n\to\infty}\|A^nx_0\|\to+\infty$.

From $(1)(2)(3)$ and $\mathbb{H}=\bigoplus\limits_{j=0}^{+\infty}\mathbb{H}_j$,
we get that for any given $0\neq x_0=(x_1,x_2,\cdots,)\in\mathbb{H}$ we can find $m\in\mathbb{N}$ such that $y_m\neq0$.
Hence for any given $0\neq x_0=(x_1,x_2,\cdots,)\in\mathbb{H}$ we get
$\varliminf\limits_{n\to\infty}\|A^nx_0\|=+\infty$.
Therefor $(I+K_\varepsilon)^{-1}$ is not Li-Yorke chaotic.
From the view of infinite matrix,
$(I+K_\varepsilon)^{*}$ and $(I+K_\varepsilon)^{*-1}$ are lower triangular matrix,
for any $0\neq x\in\mathbb{H}$,with a simple computing we get
$\varliminf\limits_{n\to\infty}\|(I+K_\varepsilon)^{*n}x\|>0$,
$\varliminf\limits_{n\to\infty}\|(I+K_\varepsilon)^{*-n}x\|>0$.
Hence $(I+K_\varepsilon)^{*}$ and $(I+K_\varepsilon)^{*-1}$ are not Li-Yorke chaotic.
\end{proof}

\begin{corollary}\label{fenbuhundundenibushiliyorkehundunde17}
There is a invertible bounded linear operator $I+K_\varepsilon$ on $\mathbb{H}=\mathcal{L}^2(\mathbb{N})$ such that $I+K_\varepsilon$ is distributionally chaotic,
but $(I+K_\varepsilon)^{-1}$,$(I+K_\varepsilon)^{*-1}$ and $(I+K_\varepsilon)^{*}$ are not distributionally chaotic.
\end{corollary}
\begin{proof}
By the construction of Theorem $\ref{duijiaoxianyujxiaojinsuanzidingliyinyong13}$,
it is only to give the conclusion by induction on $\{k_i\}_{i=1}^{\infty}$ as  Example $\ref{duijiaoxianyujxiaojinsuanzidenibuhundun14}$.
\end{proof}

\begin{theorem}\label{cunzaisuanzishuzhishiyiduankaiyuanshu7}
There is $T\in\mathcal{L}(\mathbb{H})$,
$S_{LY}(T)=S_{DC}(T)=\omega$ is an open arc of $\mathbb{T}=\{|\lambda|=1;\lambda\in\mathbb{C}\}$,
and for $\forall \lambda\in\omega$,we get that
$(\lambda+T)^{*}$,$(\lambda+T)^{*-1}$ and $(\lambda+T)^{-1}$ are not Li-Yorke chaotic.
\end{theorem}
\begin{proof}
As Example $\ref{duijiaoxianyujxiaojinsuanzidenibuhundun14}$,give the same $\mathbb{H}=\bigoplus\limits_{j=1}^{+\infty}\mathbb{H}_j$,$S_j$ and $K_j$,
give positive sequence $\{\varepsilon_j\}_{1}^{\infty}$ such that
$\varepsilon_j\to0$ and $\sup\limits_{j\to\infty}|i+\varepsilon_j|^{j}\to+\infty$,where $i\in\mathbb{C}$.

Let $\lambda I+K_\varepsilon=\bigoplus\limits_{j=1}^{+\infty}(\lambda I+K_j)$,
so $\sigma(\lambda I+K_\varepsilon)=\{\lambda -\varepsilon_j;j\in\mathbb{N}\}$.

$(i)$ If $|\lambda|<1$,because of $\varepsilon_j\to0$,
we get that $\exists N>0$ when $n>N$,
$|\lambda -\varepsilon_j|<1$.
With the introduction of this paper we get that Li-Yorke chaos is valid only on infinite Hilbert space.
Loss no generally,for any $j\in\mathbb{N}$,let $|\lambda -\varepsilon_j|<1$,
so $r_{\sigma(\lambda I+K_\varepsilon)}<1$.
Hence for any $0\neq x\in\mathbb{H}$ there is
$\lim\limits_{n\to\infty}\|(I+K_\varepsilon)^n(x)\|=0$.

$(ii)$ If $|\lambda|>1$ or $\lambda\in[\frac{\pi}{2},\frac{3\pi}{2}]$,
because of $\varepsilon_j>0$,
for $j\in\mathbb{N}$ we get $|\lambda -\varepsilon_j|>1$,
$\frac{1}{|\lambda -\varepsilon_j|}<1$,
and $\sigma(\lambda I+K_\varepsilon)^{-1}=\{\frac{1}{\lambda -\varepsilon_j};j\in\mathbb{N}\}$.
By Example $\ref{duijiaoxianyujxiaojinsuanzidenibuhundun14}$,
for any given $x\neq0$ there is $y_m\neq0,m\in\mathbb{N}$,hence we get

$\left.\begin{array}{l}
\lim\limits_{n\to\infty}\|(\lambda I+K_\varepsilon)^nx_0\|\\
\geq\lim\limits_{n\to\infty}\|(\lambda I_m+K_m)^ny_m\|\\
\geq\lim\limits_{n\to\infty}\frac{1}{\|(\lambda I_m+K_m)^{-n}\|}\|y_m\|\\
\geq\lim\limits_{n\to\infty}\frac{1}{\|(\lambda I_m+K_m)^{-1}\|^n}\|y_m\|\\
\geq\|y_m\|>0.
\end{array}\right.$

$(iii)$ For $\forall \lambda\in(-\frac{\pi}{2},\frac{\pi}{2})$,
because of $ \varepsilon_j\to0$,
there exists $N>0$, when $j>N$, we get $|\lambda-\varepsilon_j|<1$.
Let $\mathbb{H}^{'}=\bigoplus\limits_{j>N}\mathbb{H}_j$,
$(\lambda I+K_\varepsilon)^{'}=(\lambda I+K_\varepsilon)|_{\bigoplus\limits_{j>N}\mathbb{H}_j}$,
then for $f_j=\frac{1}{\sqrt{j}}(1_1,\cdots,1_j)$ and $f_{j,n}=\frac{1}{\sqrt{j}}(1_1,\cdots,1_n,0,\cdots,0)\in\mathbb{H}^{'}$,
we get that when $1\leq n\leq j$,there is

$\left.
   \begin{array}{l}
\|(\lambda I+K_\varepsilon)^n(f_j)\|\\
=\|(\lambda I_j+K_j)^n(f_j)\|\\
=\|((\lambda-\varepsilon_j)I_j+S_j)^n(f_j)\|\\
\geq\|\sum\limits_{k=0}^{n}C_{n}^{k}(\lambda-\varepsilon_j)^{k}(2\varepsilon_j)^{n-k}(1_1,\cdots,1_n,0,\cdots,0)\|\\
=|\lambda+\varepsilon_j|^n\|f_{j,n}\|.
\end{array}\right.$

By $(i)(ii)$,if $|\lambda|\neq1$ or $\lambda\in[\frac{\pi}{2},\frac{3\pi}{2}]$,
$\lambda I+K_\varepsilon$ is not Li-Yorke chaotic.

By $(iii)$ and by the property of the triangle,if $\lambda\in(-\frac{\pi}{2},\frac{\pi}{2})$ and  $j>N$,
we get $|\lambda+\varepsilon_j|>|i+\varepsilon_j|$ and $r_{\sigma}((\lambda I+K_\varepsilon)^{'})<1$.
Hence we get

$\lim\limits_{n\to\infty}\|(\lambda I+K_\varepsilon)^{'n}(f_j)\|=0$, and

$\varlimsup\limits_{j\to\infty}\|(I+K_\varepsilon)^{'j}(f_j)\|$
$\geq\lim\limits_{j\to\infty}\|f_j\||\lambda+\varepsilon_j|^j\geq\lim\limits_{j\to\infty}|i+\varepsilon_j|^j=+\infty$.

By Definition $\ref{liyorkechaoscriteriondingyi0}$ we get that $\lambda I+K_\varepsilon$ satisfies the Li-Yorke Chaos Criterion,
by Theorem $\ref{liyorkechaoscriteriondingli0}$ we get that $\lambda I+K_\varepsilon$ is Li-Yorke chaotic.

Using the same proof of Example $\ref{duijiaoxianyujxiaojinsuanzidenibuhundun14}$ we get that
$(\lambda I+K_\varepsilon)^{*}$,$(\lambda I+K_\varepsilon)^{*-1}$ and $(\lambda I+K_\varepsilon)^{-1}$ are not Li-Yorke chaotic.

Using Corollary $\ref{fenbuhundundenibushiliyorkehundunde17}$ and Theorem $\ref{duijiaoxianyujxiaojinsuanzidingliyinyong13}$ we get that $\lambda I+K_\varepsilon$ is distributionally chaotic,but $(\lambda I+K_\varepsilon)^{*}$,$(\lambda I+K_\varepsilon)^{*-1}$ and $(\lambda I+K_\varepsilon)^{-1}$ are not Li-Yorke chaotic.
\end{proof}


\begin{conjecture}
For any given $m\in\mathbb{N}$,
there exists $m$-folder complex analytic function $\phi(z)\in\mathcal{H}^{\infty}(\mathbb{D})$  such that $\phi$ is not a Cowen-Douglas function.
\end{conjecture}

\begin{question}
Gives the equivalent characterization of a $m$-folder complex analytic function;
Gives the equivalent characterization of a rooter function;
Gives the equivalent characterization of a Cowen-Douglas function.
Gives the relations between them.
\end{question}

\begin{question}
Let $M_{\phi}$ is the multiplication operator of the Cowen-Douglas function $\phi(z)\in\mathcal{H}^{\infty}(\mathbb{D})$ on the Hardy space $\mathcal{H}^{2}(\mathbb{D})$,
then is $M_{\phi}^{*}$ a Lebesgue operator? If not and if they have relations ,gives the relations between them.
\end{question}

\end{document}